\documentclass[12pt, a4paper, reqno]{amsart}
\usepackage{amsmath}
\usepackage{amsfonts}
\usepackage{amssymb}
\usepackage{color}
\usepackage{comment,graphicx}
\usepackage[T1]{fontenc}
\usepackage{url}
\usepackage{ae}
\usepackage{array}

\usepackage{xcolor}
\usepackage{hyperref}
\hypersetup{
colorlinks=true,
linkcolor=blue,
citecolor=blue
}

\def\phi{\varphi}
\def\rho{\varrho}
\def\epsilon{\varepsilon}
\numberwithin{equation}{section}
\theoremstyle{plain}
\newtheorem{theorem}{Theorem}[section]
\newtheorem{lemma}[theorem]{Lemma}

\newtheorem{definition}[theorem]{Definition}

\renewcommand{\leq}{\leqslant}
\renewcommand{\geq}{\geqslant}

\begin{document}

\title[
Commutator estimates for vector fields on \ $B_{p(\cdot ),q(\cdot )}^{s
(\cdot )}$ spaces]{Commutator estimates for vector fields on Besov spaces with variable smoothness and integrability }
\author[ S. BenMahmoud]{ Salah BenMahmoud}

\begin{abstract}
In this paper we present certain bilinear estimates for commutators on
Besov spaces with variable smoothness and integrability, and
under no vanishing assumptions on the divergence of vector fields. 
Such commutator estimates are motivated by the study of well-posedness results for some models in incompressible fluid mechanics.\newline


\end{abstract}

\date{December, 7. 2023 }

\keywords{Commutator. Vector fields. Besov space.  Variable exponent\\
\\
 Salah BenMahmoud \\
M'sila University, Department of Mathematics, Laboratory of Functional
Analysis and Geometry of Spaces , P.O. Box 166, M'sila 28000, Algeria.\\
e-mails:\texttt{salahmath2016@gmail.com/salah.benmahmoud@univ-msila.dz(Salah BenMahmoud).}}
\maketitle

\section{\textbf{Introduction}}
Let $ V=(V_{1},...,V_{n})$ be a smooth vector field in $\mathbb{R}^{n}$ and let $\Delta _{j}f=\varphi _{j}\ast f,j\in \mathbb{N}_{0}$ where $(\mathcal{F}\varphi _{j})_{j\in \mathbb{N}_{0}}$ is a smooth dyadic
resolution of unity, the estimates of the   commutator
\begin{equation}
\lbrack V\cdot \nabla ,\Delta _{j}]f=\sum_{k=1}^{n}V_{k}\partial _{k}\Delta
_{j}f-\Delta _{j}(V_{k}\partial _{k}f),   \label{HN-est}
\end{equation}%
is considered one of the main tools to study well-posedness, existence and uniqueness of solutions for many types of partial differential equations over  function spaces  such as  Euler equations, Navier–Stokes equations and Boussinesq system, see for
example the papers \cite{Ch02,Chen06,RD1}, the monograph  \cite{BCD}   and the references therein. The estimates are usually   proved by means of paraproducts and under the assumption that $V$ is divergence-free. 

In \cite{HN18}, the  authors developed new unifying approach to  estimate the commutator (\ref{HN-est})  over  various function spaces; weighted and variable exponent Lebesgue,Triebel-Lizorkin, and Besov spaces. 
This approach didn't use paraproducts but it was based on \cite[Lemma 3.1]{HN18}  and   duality arguments  such as the norm duality of $L^{p\left( \cdot \right)}(\ell^{q})$ and $\ell ^{q}(L^{p\left( \cdot \right)})$  stated in \cite[Lemma 6.1]{HN18}.  The estimates were obtained under no vanishing assumptions on the divergence of the vector field. In  particular, the estimates obtained on variable Triebel-Lizorkin and Besov spaces where restricted to the scales $F_{p(\cdot ),q}^{s}$ and $B_{p(\cdot ),q}^{s}$ with only constant  indices $q$ and $s$ , not variable functions. this is mainly due to employing the maximal operator which  is bounded only on $L^{p\left( \cdot \right)}(\ell^{q})$ and $\ell ^{q}(L^{p\left( \cdot \right)})$ when $q$ is a constant and $p$  satisfies certain requirements.

Later, to obtain more general estimates on   Triebel-Lizorkin spaces with variable smoothness and integrability
$F_{p(\cdot ),q(\cdot)}^{s(\cdot)}$ and allow all the indices to be variables and since 
the maximal operator is not bounded on  $L^{p\left( \cdot \right)}(\ell^{q(\cdot)})$  the authors in 
\cite{art1} pursued a  different approach to overcome this difficulty and others through the use of  
\cite[Lemma 2]{art1}, obtain a more generalized assertion to \cite[Lemma 6.1]{HN18}  introduced in   \cite[Lemma 4]{art1}, and impose some regularity assumptions on the indices. In this paper, Lemmas 
\ref{equa thm1} and \ref{duality2} are the corresponding lemmas to 
\cite[Lemma 2]{art1}  and 
\cite[Lemma 4]{art1} for $\ell ^{q(\cdot)}(L^{p\left( \cdot \right)})$, respectively.

 Our main goals  in this paper are; first, prove the  duality argument presented in Lemma \ref{duality2} which
   is an important result for the variable scales $\ell ^{q(\cdot)}(L^{p\left( \cdot \right)})$, it is 
useful to deal with the complicated  norm of this scales in a different way, the second goal is estimating  the commutator (\ref{HN-est}) on 
$B_{p(\cdot ),q(\cdot)}^{s(\cdot)}$ and generelize the results of \cite{HN18} so that all the indices $s,p$ and $q$ can be functions, the results here cover  the cases where $p^-=1$ or $p^+=\infty$, also,  $p$ or $q$ can be $\infty$ on some subsets of $\mathbb{R}^n$, these last introduces further complications and extra challenges. The proofs are 
written clearly and each step is  explained well that is easy to follow and understand.

The remainder of this paper is organized as follows. In Section \ref{Sect2}, we set some notation and present definitions and basic results about Besov spaces with variable smoothness and integrability.
 The Section \ref{secc1} is focused on presenting the norm duality of  $\ell ^{q(\cdot)}(L^{p(\cdot )})$, we prove the generalization of \cite[Lemma 6.1]{HN18} which was stated only for 
$\ell ^{q}(L^{p(\cdot )})$ where $q\in[1,\infty]$ is constant and $p$ is a bounded variable exponent with $p^->1$.   In  Section \ref{Mainsec} we begin by proving preliminary lemmas and then we employ them to prove the  main results, Theorems \ref{result1}, \ref{result2} and \ref{result3}.
\section{\textbf{Preliminaries}}
\label{Sect2}
Now, we present some notations. As usual, we denote by $\mathbb{R}^{n}$ the $n$-dimensional real Euclidean
space, $\mathbb{N}$ the collection of all natural numbers and $\mathbb{N}%
_{0}=\mathbb{N}\cup \{0\}$. For a multi-index $\alpha =(\alpha
_{1},...,\alpha _{n})\in \mathbb{N}_{0}^{n}$, we write $\left\vert \alpha
\right\vert =\alpha _{1}+...+\alpha _{n}$. The notation $f\lesssim g$ means that $%
f\leq c\,g$ for some independent positive constant $c$ (and non-negative
functions $f$ and $g$), and $f\approx g$ means that $f\lesssim g\lesssim f$.

If $E\subset {\mathbb{R}^{n}}$ is a measurable set, then $|E|$ stands for
the Lebesgue measure of $E$ and $\chi _{E}$ denotes its characteristic
function. By $c$ we denote generic positive constants, which may have
different values at different occurrences. Although the exact values of the
constants are usually irrelevant for our purposes, sometimes we emphasize
their dependence on certain parameters (e.g., $c(p)$ means that $c$ depends
on $p$, etc.).

  Let $\mathbf{f}=(f_{1},...,f_{n})\in X^{n}$ for some normed space $X$. Then we put%
$\big\|\mathbf{f}\big\|_{X}=\sum_{i=1}^{n}\big\|f_{i}\big\|_{X}$.

The symbol $\mathcal{S}(\mathbb{R}^{n})$ is used in place of the set of all
Schwartz functions on $\mathbb{R}^{n}$. We define the Fourier transform of a
function $f\in \mathcal{S}(\mathbb{R}^{n})$ by 
\begin{equation*}
\mathcal{F}(f)(\xi ):=(2\pi )^{-n/2}\int_{\mathbb{R}^{n}}e^{-ix\cdot \xi
}f(x)dx,\quad \xi \in \mathbb{R}^{n}.
\end{equation*}
The Hardy-Littlewood maximal operator $\mathcal{M}$ is defined for a locally integrable function $ f \in L_{%
\mathrm{loc}}^{1}$ by
\begin{equation*}
\mathcal{M}f(x)=\sup_{r>0}\frac{1}{\left\vert B(x,r)\right\vert }%
\int_{B(x,r)}\left\vert f(y)\right\vert dy.
\end{equation*}

We denote by $\mathcal{S}^{\prime }(\mathbb{R}^{n})$ the dual space of all
tempered distributions on $\mathbb{R}^{n}$. The variable exponents that we
consider are always measurable functions $p$ on $\mathbb{R}^{n}$ with range
in $[1,\infty ]$, we denote the set of all such
functions by $\mathcal{P}(\mathbb{R}^{n})$. For $p\in \mathcal{P}(\mathbb{R}^{n})$
the conjugate exponent of $p$ denoted by ${p}^{\prime }$ 
 is  given by $\frac{1}{{p(\cdot )}}+\frac{1}{{p}%
^{\prime }{(\cdot )}}=1$ with the convention  $\frac{1}{\infty} = 0$.
 We use the standard notations: 
\begin{equation*}
p^{-}:=\underset{x\in \mathbb{R}^{n}}{\text{ess-inf}}\,p(x)\quad \text{and}%
\quad p^{+}:=\underset{x\in \mathbb{R}^{n}}{\text{ess-sup}}\,p(x).
\end{equation*}%
The function spaces in this paper are fit into the framework of  semi-modular spaces, see for example \cite[Chapter 2]{DHHR} and \cite{Orli-spa}. The function $\omega _{p}$ is defined as follows:
\begin{equation*}
\omega _{p}(t)=\left\{ 
\begin{array}{ccc}
t^{p} & \text{if} & p\in \lbrack 1,\infty )\text{ and }t>0, \\ 
0 & \text{if} & p=\infty \text{ and }0<t\leq 1, \\ 
\infty & \text{if} & p=\infty \text{ and }t>1.%
\end{array}%
\right.
\end{equation*}%
The convention $1^\infty=0$ is adopted in order that $\omega _{p}$ be left-continuous. The variable exponent modular is defined by 
\begin{equation*}
\varrho _{p(\cdot )}(f):=\int_{\mathbb{R}^{n}}\omega _{p(x)}(|f(x)|)\,dx.
\end{equation*}%
The variable exponent Lebesgue space $L^{p(\cdot )}$\ consists of measurable
functions $f$ on $\mathbb{R}^{n}$ such that $\varrho _{p(\cdot )}(\lambda
f)<\infty $ for some $\lambda >0$. We define the Luxemburg (quasi)-norm on
this space by the formula 
\begin{equation*}
\left\Vert f\right\Vert _{p(\cdot )}:=\inf \Big\{\lambda >0:\varrho
_{p(\cdot )}\Big(\frac{f}{\lambda }\Big)\leq 1\Big\}.
\end{equation*}%
We have  $\left\Vert f\right\Vert _{p(\cdot )}\leq 1$ if
and only if $\varrho _{p(\cdot )}(f)\leq 1$, see \cite[Lemma 3.2.4]{DHHR}. By \cite[Lemma 3.2.8]{DHHR},  for a sequence of measurable functions $(f_{j})_{j\in \mathbb{N}_{0}}$ and a measurable function $f$ if $|f_j|\nearrow |f|$, then
\begin{align}\label{pcv-eq}
\varrho _{p(\cdot )}(f)=\lim_j \varrho _{p(\cdot )}(f_j).
\end{align}

Let $p,q\in \mathcal{P}(\mathbb{R}^{n})$. The mixed Lebesgue-sequence space $\ell ^{q(\cdot )}(L^{p(\cdot )})$ is defined on sequences of $%
L^{p(\cdot )}$-functions by the modular 
\begin{equation*}
\varrho _{\ell ^{q(\cdot )}(L^{p\left( \cdot \right)
})}((f_{j})_{j\in \mathbb{N}_{0}}):=\sum_{j=0}^{\infty }\inf \Big\{\lambda _{j}>0:\varrho
_{p(\cdot )}\Big(\frac{f_{j}}{\lambda _{j}^{1/q(\cdot )}}\Big)\leq 1\Big\},
\end{equation*}%
with the convention  $\lambda^{1/\infty} = 1$. The (quasi)-norm is defined from this as usual:%
\begin{equation}
\left\Vert \left( f_{j}\right) _{j\in \mathbb{N}_{0}}\right\Vert _{\ell ^{q(\cdot
)}(L^{p\left( \cdot \right) })}:=\inf \Big\{\mu >0:\varrho _{\ell ^{q(\cdot
)}(L^{p(\cdot )})}\Big(\frac{1}{\mu }(f_{j})_{j\in \mathbb{N}_{0}}\Big)\leq 1\Big\}.
\label{mixed-norm}
\end{equation}%
 In particular, if $p(\cdot)=\infty $, then we can replace $\mathrm{\eqref{mixed-norm}}$ by the expression
\begin{equation}
\varrho _{\ell ^{q(\cdot )}(L^{\infty})}((f_{j})_{j\in \mathbb{N}_{0}})=\sum_{j=0}^{\infty
} \underset{x\in \mathbb{R}^{n}}{\text{ess-sup}}\, |f_{j}(x)|^{q(x )}, \label{ess sup mod}
\end{equation}%
and the case    $q(x)=\infty$ is included by the convention $t^\infty = \infty\chi_{]1;\infty]}(t)$. 
If $q(\cdot)=\infty$ then 
\begin{equation}
\left\Vert \left( f_{j}\right) _{j\in \mathbb{N}_{0}}\right\Vert _{\ell ^{\infty}(L^{p\left( \cdot \right) })}=
\sup_{j\in \mathbb{N}_{0}}\big\| f_j \big\|_{p(\cdot)}.
\label{q=infty norm}
\end{equation}%

We recall some  useful properties, we have  $\| \left( f_{j}\right) _{j\in \mathbb{N}_{0}}\| _{\ell ^{q(\cdot
)}(L^{p\left( \cdot \right) })}\leq 1$ if
and only if $\varrho _{\ell ^{q(\cdot )}(L^{p\left( \cdot \right)
})}((f_{j})_{j\in \mathbb{N}_{0}}) \leq 1$. The first property (i) of the following lemma is from \cite{AH} while the second (ii) can be proven easily by  (\ref{ess sup mod}).
\begin{lemma}\label{cont-lem} For $p,q\in \mathcal{P}(\mathbb{R}^{n})$, we have\\ 
(i) if $p^+,q^+ <\infty$ then the function $\mu:]0;+\infty[ \to \varrho _{\ell ^{q(\cdot )}(L^{p\left( \cdot \right)
})}((f)_{j\in \mathbb{N}_{0}}/\mu)$ is continuous for every $(f)_{j\in \mathbb{N}_{0}} \in \ell ^{q(\cdot )}(L^{p\left( \cdot \right)})$;\\
(ii) the function $\mu:]0;+\infty[ \to \varrho _{\ell ^{q(\cdot )}(L^{\infty})}((f)_{j\in \mathbb{N}_{0}}/\mu)$ is continuous  when $q^+ <\infty$ for every $(f)_{j\in \mathbb{N}_{0}} \in \ell^{q(\cdot )}(L^{\infty})$.
\end{lemma}

We say that a real valued-function $g$ on $\mathbb{R}^{n}$ is \textit{%
locally }log\textit{-H\"{o}lder continuous} on $\mathbb{R}^{n}$, abbreviated 
$g\in C_{\text{loc}}^{\log }(\mathbb{R}^{n})$, if there exists a constant $%
c_{\log }(g)>0$ such that 
\begin{equation}
\left\vert g(x)-g(y)\right\vert \leq \frac{c_{\log }(g)}{\log
(e+1/\left\vert x-y\right\vert )}  \label{lo-log-Holder}
\end{equation}%
for all $x,y\in \mathbb{R}^{n}$.

We say that $g$ satisfies the log\textit{-H\"{o}lder decay condition}, if
there exist two constants $g_{\infty }\in \mathbb{R}$ and $c_{\log }>0$ such
that%
\begin{equation*}
\left\vert g(x)-g_{\infty }\right\vert \leq \frac{c_{\log }}{\log
(e+\left\vert x\right\vert )}
\end{equation*}%
for all $x\in \mathbb{R}^{n}$. We say that $g$ is \textit{globally} log%
\textit{-H\"{o}lder continuous} on $\mathbb{R}^{n}$, abbreviated $g\in
C^{\log }(\mathbb{R}^{n})$, if it is\textit{\ }locally log-H\"{o}lder
continuous on $\mathbb{R}^{n}$ and satisfies the log-H\"{o}lder decay\textit{%
\ }condition.\textit{\ }The constants $c_{\log }(g)$ and $c_{\log }$ are
called the \textit{locally }log\textit{-H\"{o}lder constant } and the log%
\textit{-H\"{o}lder decay constant}, respectively\textit{.} We note that any
function $g\in C_{\text{loc}}^{\log }(\mathbb{R}^{n})$ always belongs to $%
L^{\infty }$.\vskip5pt

We define the following class of variable exponents: 
\begin{equation*}
\mathcal{P}^{\mathrm{log}}(\mathbb{R}^{n}):=\Big\{p\in \mathcal{P}(
\mathbb{R}^{n}):\frac{1}{p}\in C^{\log }(\mathbb{R}^{n})\Big\},
\end{equation*}%
which is introduced in \cite[Section 2]{DHHMS}. We define 
\begin{equation*}
\frac{1}{p_{\infty }}:=\lim_{|x|\rightarrow \infty }\frac{1}{p(x)},
\end{equation*}%
and we use the convention $\frac{1}{\infty }=0$. Note that although $\frac{1%
}{p}$ is bounded, the variable exponent $p$ itself can be unbounded. We put 
\begin{equation*}
\Psi \left( x\right) :=\sup_{\left\vert y\right\vert \geq \left\vert
x\right\vert }\left\vert \varphi \left( y\right) \right\vert
\end{equation*}%
for $\varphi \in L^{1}$. We suppose that $\Psi \in L^{1}$. Then it is
proved in $\text{\cite[Lemma \ 4.6.3]{DHHR}}$ that if $p\in \mathcal{P}^{%
\text{log}}(\mathbb{R}^{n})$, then 
\begin{equation*}
\Vert \varphi _{\varepsilon }\ast f\Vert _{{p(\cdot )}}\leq c\Vert \Psi
\Vert _{{1}}\Vert f\Vert _{{p(\cdot )}} 
\end{equation*}%
for all $f\in L^{p(\cdot )}$, where $\varphi _{\varepsilon }(\cdot):=\varepsilon ^{-n}\varphi (
\cdot /\varepsilon ) , \varepsilon >0.$
We put $\eta _{j,m}(x):=2^{jn}(1+2^{j}\left\vert x\right\vert )^{-m}$
for any $x\in \mathbb{R}^{n}$ and $m>0$, note that when $m>n$, $\eta _{j,m}\in
L^{1}$,  $\left\Vert \eta _{j,m}\right\Vert _{1}=c(m)$ and  if $p\in \mathcal{P}^{\text{log}}(\mathbb{R}^{n})$, then
\begin{equation} 
\Vert \eta _{j,m} \ast f\Vert _{{p(\cdot )}}\leq c \Vert f\Vert _{{p(\cdot )}} \label{eneq p norm}
\end{equation}
 are independent of $j$. It was shown in \cite[Theorem 4.3.8]{DHHR}  that $\mathcal{M}:L^{p(\cdot
)}\rightarrow L^{p(\cdot )}$ is bounded if $p\in \mathcal{P}^{\mathrm{log}}$
and $p^{-}>1$, this result was widely used  in \cite{art1} and in \cite[Section 6]{HN18}, but since we aim to allow the case $p^-=1$, this is not so helpful, therefore in this paper we don't use this result,  the maximal operator  $\mathcal{M}$ is replaced by $(\eta _{j,m})_{j\in \mathbb{N}_0}$ and the previous result is replaced by the inequality (\ref{eneq p norm}) or some closely  related inequalities.
 We refer to the recent monographs \cite{CF,DHHR} for further
properties, historical remarks and references on variable exponent Lebesgue spaces.

To define Besov spaces with variable smoothness and integrability,
let us first introduce the concept of a smooth dyadic
resolution of unity or dyadic decomposition of unity, see  \cite[Section 2.3.1]{T1}. Let $\Phi $\ be a function\ in $\mathcal{S}(%
\mathbb{R}
^{n})$\ satisfying $\Phi (x)=1$\ for\ $\left\vert x\right\vert \leq 1$\ and\ 
$\Phi (x)=0$\ for\ $\left\vert x\right\vert \geq 2$.\ We define  $\phi_0 $ and $\phi$ by $\mathcal{F}\phi_0=\Phi$ and $\mathcal{F}\phi
(x)=\Phi (x)-\Phi (2x)$\ and 
\begin{equation*}
\mathcal{F}\phi _{j}(x)=\mathcal{F}\phi (2^{-j}x)\quad \text{\textit{for}%
}\quad j \in \mathbb{N}.
\end{equation*}%
Then $\{\mathcal{F}\phi _{j}\}_{j\in \mathbb{N}_{0}}$\ is a smooth dyadic
resolution of unity, that is
\begin{itemize}
\item[\textit{(i)}] $\mathrm{supp}\,\mathcal{F}\phi_0 \subset \{x\in \mathbb{R}^{n}:\left\vert
x\right\vert \leq 2\}$;
\item[\textit{(ii)}] $\mathrm{supp}\,\mathcal{F}\phi \subset \big\{%
x\in \mathbb{R}^{n}:1/2\leq \left\vert x\right\vert \leq 2\big\} $
; and
\item[ \textit{(iii)}]  $\sum_{j =0}^{\infty }\mathcal{F}\phi _{j}(x)=1$ for all 
$x\in \mathbb{R}^{n}$,
\end{itemize}
any system of functions $\{\phi_{j},j \in \mathbb{N}_0\}\subset \mathcal{S}(\mathbb{R}^n)$ satisfies\textit{(i),(ii)} and \textit{(iii)} is called smooth dyadic
resolution of unity.
  Thus we obtain the Littlewood-Paley decomposition 
\begin{align*}
f=\sum_{j=0}^{\infty }\phi _{j}\ast f
\end{align*}%
for all $f\in \mathcal{S}^{\prime }(%
\mathbb{R}
^{n})$\ $($convergence in $\mathcal{S}^{\prime }(%
\mathbb{R}
^{n}))$.

We state the definition of the spaces $B_{p(\cdot ),q(\cdot )}^{s(\cdot )}$,
which was introduced and studied in~\cite{AH}.
\begin{definition}
\label{11}\textit{Let }$\left\{ \mathcal{F}\phi _{j}\right\} _{j\in \mathbb{N%
}_{0}}$\textit{\ be a resolution of unity}, $s:\mathbb{R}^{n}\rightarrow 
\mathbb{R}$ and $p,q\in \mathcal{P}(\mathbb{R}^{n})$.\textit{\ The Besov
space }$B_{p(\cdot ),q(\cdot )}^{s(\cdot )}$\textit{\ consists of all
distributions }$f\in \mathcal{S}^{\prime }(%
\mathbb{R}
^{n})$\textit{\ such that}%
\begin{equation*}
\left\Vert f\right\Vert _{B_{p(\cdot ),q(\cdot )}^{s(\cdot )}}:=\left\Vert
(2^{j s(\cdot )}\phi _{j}\ast f)_{j\in \mathbb{N}_0}\right\Vert _{\ell ^{q(\cdot
)}(L^{p\left( \cdot \right) })}<\infty .
\end{equation*}
\end{definition}

If we take $s\in \mathbb{R}$ and $q\in [1,\infty ]$ as constants,  the
spaces $B_{p(\cdot ),q}^{s}$ where studied  by Xu in \cite{X1}. We refer the
reader to the recent papers  \cite{AlCa152}, \cite{D3} and 
\cite{KV122} for further details, historical remarks and more references on
these function spaces. For any $p,q\in \mathcal{P}^{\log }(\mathbb{R}%
^{n})$ and $s\in C_{\text{loc}}^{\log }$, the space $B_{p(\cdot ),q(\cdot
)}^{s(\cdot )}$ does not depend on the chosen smooth dyadic resolution of 
unity $\{\mathcal{F}\phi _{j}\}_{j\in \mathbb{N}_{0}}$ (in the
sense of\ equivalent quasi-norms). Moreover, if $p,q,s$ are constants, we re-obtain the usual Besov spaces $%
B_{p,q}^{s}$, studied in detail in \cite{T1,T2,T4}.

Now we recall  the following  lemmas. We begin by \cite[Lemma 6.1]{DHR}, see also \cite[Lemma 19]{KV122} 
\begin{lemma}
\label{DHR-lemma}Let $\alpha \in C_{\mathrm{loc}}^{\log }(\mathbb{R}^{n})$ and let $R\geq
c_{\log }(\alpha )$, where $c_{\log }(\alpha )$ is the constant from $%
\mathrm{\eqref{lo-log-Holder}}$ for $\alpha $. Then%
\begin{equation*}
2^{j\alpha (x)}\eta _{j,m+R}(x-y)\leq c\text{ }2^{j\alpha (y)}\eta
_{j,m}(x-y)
\end{equation*}%
with $c>0$ independent of $x,y\in \mathbb{R}^{n}$ and $j,m\in \mathbb{N}%
_{0}. $
\end{lemma}

The previous lemma allows us to treat the variable smoothness in many cases
as if it were not variable at all, namely we can move the term inside the
convolution as follows:%
\begin{equation*}
2^{j\alpha (x)}\eta _{j,m+R}\ast f(x)\leq c\text{ }\eta _{j,m}\ast
(2^{j\alpha (\cdot )}f)(x).
\end{equation*}

Since the maximal operator is in general not bounded on $\ell ^{q(\cdot )}(L^{p(\cdot )})$, see \cite[Section 4]{AH},
the following statement is of great interest in this paper, see  \cite[Lemma 4.7]{AH}.
\begin{lemma}
\label{Alm-Hastolemma1} Let $p\in \mathcal{P}^{\log }(\mathbb{R}^{n})$ and $%
q\in \mathcal{P}\left( \mathbb{R}^{n}\right) $\ with\ $\frac{1}{q}\in C_{%
\mathrm{loc}}^{\log }\left( \mathbb{R}^{n}\right) $. For $m>n+c_{\log }(1/q)$%
, there exists $c>0$ such that%
\begin{equation*}
\left\Vert \left( \eta _{j,m}\ast f_{j}\right) _{j\in \mathbb{N}_{0}}\right\Vert _{\ell
^{q(\cdot )}(L^{p(\cdot )})}\leq c\left\Vert \left( f_{j}\right)
_{j\in \mathbb{N}_{0}}\right\Vert _{\ell ^{q(\cdot )}(L^{p(\cdot )})}.
\end{equation*}
\end{lemma}
\section{\textbf{The Norm Duality Of  $\ell ^{q(\cdot)}(L^{p(\cdot )})$}}
\label{secc1}

\begin{lemma} \label{Thm inc}
Let $p,q\in \mathcal{P}(\mathbb{R}^{n})$ and $\{ (f_{j}^n)_{j\in \mathbb{N}_{0}}\}_{n\in \mathbb{N}_0}$ be a sequence of elements of $\ell ^{q(\cdot
)}(L^{p(\cdot )})$, suppose that  $|f_{j}^n|\leq |f_{j}^{n+1}|$ and $\lim_n f_j^n(x) =f_j(x)$ for all $j, n \in \mathbb{N}_0 \text{ and } x \in \mathbb{R}^n$, then 
\begin{equation}\label{equa thm1}
\sup_n \big\|(f_{j}^n)_{j\in \mathbb{N}_{0}}\big\|_{\ell ^{q(\cdot
)}(L^{p(\cdot )})}=\big\|(f_{j})_{j\in \mathbb{N}_{0}}\big\|_{\ell ^{q(\cdot
)}(L^{p(\cdot )})}.
\end{equation}
\end{lemma}
\begin{proof}
It's clear that the left hand side of (\ref{equa thm1}) is increasing and less than the right hand side, thus
\begin{align*}
\sup_n \big\|(f_{j}^n)_{j\in \mathbb{N}_{0}}\big\|_{\ell ^{q(\cdot
)}(L^{p(\cdot )})}\leq\big\|(f_{j})_{j\in \mathbb{N}_{0}}\big\|_{\ell ^{q(\cdot
)}(L^{p(\cdot )})}.
\end{align*}
Now, we prove the reverse inequality, i.e.,
\begin{align*}
\sup_n \big\|(f_{j}^n)_{j\in \mathbb{N}_{0}}\big\|_{\ell ^{q(\cdot
)}(L^{p(\cdot )})}\geq\big\|(f_{j})_{j\in \mathbb{N}_{0}}\big\|_{\ell ^{q(\cdot
)}(L^{p(\cdot )})}.
\end{align*}
Denote $\mu_n:=\big\|(f_{j}^n)_{j\in \mathbb{N}_{0}}\big\|_{\ell ^{q(\cdot)}(L^{p(\cdot )})}$, if $\sup_n \mu_n=\infty$ its clear that (\ref{equa thm1}) holds, on the other hand,  let  $K:=\sup_n \mu_n+\delta  \text{ where } \delta>0$ and 
\begin{align*}
\beta_j^n:=\inf \Big\{\lambda _{j}^n>0:\varrho
_{p(\cdot )}\Big(\frac{f_{j}^n}{K\lambda _{j}^{n\, 1/q(\cdot )}}\Big)\leq 1\Big\},
\end{align*}
it follows, for any $n \in \mathbb{N}_{0}$,  $ \sum_{j=0}^{\infty } \beta_j^n\leq 1$ and for all $j,n \in \mathbb{N}_{0}$, $\beta_j^n \leq \beta_j^{n+1}$ . Let $\beta_j:=\lim_n \beta_j^n = sup_n \beta_j^n$, hence, we have $ \sum_{j=0}^{\infty } \beta_j\leq 1$. Let $\gamma_j:=\beta_j+\varepsilon/2^{j-1}$ for every $j \in \mathbb{N}_{0}$ and an $\varepsilon>0$, thus, $ \sum_{j=0}^{\infty } \gamma_j\leq 1+\varepsilon$ and by (\ref{pcv-eq}),
\begin{align*}
\varrho_{p(\cdot )}\Big(\frac{f_{j}}{K\gamma_{j}^{1/q(\cdot )}}\Big)&=\lim_n
\varrho_{p(\cdot )}\Big(\frac{f_{j}^n}{K\gamma_{j}^{1/q(\cdot )}}\Big)\\
&\leq \lim_n \varrho_{p(\cdot )}\Big(\frac{f_{j}^n}{K(\beta_j^n+\varepsilon/2^{j-1})^{1/q(\cdot )}}\Big)\\
&\leq  1,
\end{align*}
it follows that 
\begin{align*}
\sum_{j=0}^{\infty }\inf \Big\{\lambda _{j}>0:\varrho
_{p(\cdot )}\Big(\frac{f_{j}}{K\lambda _{j}^{ 1/q(\cdot )}}\Big)\leq 1\Big\} \leq 
\sum_{j=0}^{\infty } \gamma_j\leq 1+\varepsilon,
\end{align*}
since $\varepsilon$ is arbitrary we conclude that $\varrho _{\ell ^{q(\cdot
)}(L^{p(\cdot )})}((f_{j})_{j\in \mathbb{N}_{0}}/K)\leq 1$. Therefore, $(f_{j})_{j\in \mathbb{N}_{0}} \in \ell ^{q(\cdot
)}(L^{p(\cdot )})$  and
\begin{align*}
\big\|(f_{j})_{j\in \mathbb{N}_{0}}\big\|_{\ell ^{q(\cdot
)}(L^{p(\cdot )})}\leq \sup_n \big\|(f_{j}^n)_{j\in \mathbb{N}_{0}}\big\|_{\ell ^{q(\cdot)}(L^{p(\cdot )})} +\delta,
\end{align*}
letting $\delta$ go to zero we get the reverse inequality, which completes the proof.
\end{proof}

The next Lemma which is some times called \textit{the norm conjugate formula} is \cite[Corollary 3.2.14]{DHHR},
see also \cite{CF}.
\begin{lemma}\label{Norm cong Form} Let $p \in  \mathcal{P}(\mathbb{R}^{n})$. Then 
$$\big\|f\big\|_{p(\cdot )} \approx \sup_{\|g\|_{p'(\cdot )}\leq 1} \int |f| |g| .$$
\end{lemma}

The following Lemma is the main result of this section.
\begin{lemma}
\label{duality2} Let $p,q\in \mathcal{P}(\mathbb{R}^{n})$ and $\left( f_{j}\right) _{j\in \mathbb{N
}_{0}}\in\ell ^{q(\cdot)}(L^{p(\cdot )})$. Then
\begin{equation*}
\big\|(f_{j})_{j\in \mathbb{N}_{0}}\big\|_{\ell ^{q(\cdot
)}(L^{p(\cdot )})}\approx \sup_{( g_{j}) _{j\in \mathbb{N
}_{0}} \in U_{p',q'}} \int_{\mathbb{R}^{n}}\sum_{j=0}^{\infty }|f_{j}(x)||g_{j}(x)|dx,
\end{equation*}%
where $U_{p',q'}$ is the unit ball centered at zero in $\ell ^{q'(\cdot
)}(L^{p'(\cdot )})$ 
of all functions $\left(
g_{j}\right) _{j\in \mathbb{N}_{0}}\in \ell ^{q'(\cdot
)}(L^{p'(\cdot )})$ such that 
\begin{equation*}
\big\|(g_{j})_{j\in \mathbb{N}_{0}}\big\|_{\ell ^{q'(\cdot
)}(L^{p'(\cdot )})}\leq 1.
\end{equation*}
\end{lemma}
\begin{proof}\quad\\
\textbf{Step 1:} First, we prove 
\begin{equation*}
\sup_{( g_{j}) _{j\in \mathbb{N
}_{0}} \in U_{p',q'}} \int_{\mathbb{R}^{n}}\sum_{j=0}^{\infty }|f_{j}(x)||g_{j}(x)|dx \lesssim \big\|(f_{j})_{j\in \mathbb{N}_{0}}\big\|_{\ell ^{q(\cdot
)}(L^{p(\cdot )})},
\end{equation*}
by scaling arguments, it suffices to prove that 
 \begin{equation}\label{st1-ene}
\sup_{( g_{j}) _{j\in \mathbb{N
}_{0}} \in U_{p',q'}} \int_{\mathbb{R}^{n}}\sum_{j=0}^{\infty }|f_{j}(x)||g_{j}(x)|dx \leq c
\end{equation}
where $\|(f_{j})_{j\in \mathbb{N}_{0}}\|_{\ell ^{q(\cdot
)}(L^{p(\cdot )})}\leq1$. By definition of $\|\cdot\|_{\ell ^{q(\cdot
)}(L^{p(\cdot )})}$, there exist positive constants $\lambda_j,\beta_j,j\in~\mathbb{N}_0$ such that 
$\|\lambda_j^{-1/q(\cdot)}f_{j}\|_{p(\cdot)}\leq 1$, $\|\beta_j^{-1/q'(\cdot)}g_{j}\|_{p'(\cdot)}\leq 1$, $\sum_{j=0}^{\infty }\lambda_j\leq 2$ and $\sum_{j=0}^{\infty }\beta_j\leq 2,$ set $K_j :=\max\{\lambda_j,\beta_j\}$. Since $K_j\geq \lambda_j,K_j\geq \beta_j$ then $\|K_j^{-1/q(\cdot)}f_{j}\|_{p(\cdot)}\leq 1$ and  $\|K_j^{-1/q'(\cdot)}g_{j}\|_{p'(\cdot)}\leq 1$,   by H\"{o}lder's  inequality we have
\begin{equation*}
 \int_{\mathbb{R}^{n}} |f_{j}(x)||g_{j}(x)|dx \leq K_j \int_{\mathbb{R}^{n}} K_j^{-1/q(\cdot)}|f_{j}(x)| K_j^{-1/q'(\cdot)} |g_{j}(x)|dx \leq c K_j,
\end{equation*}
therefore
\begin{equation*}
\sum_{j=0}^{\infty } \int_{\mathbb{R}^{n}}|f_{j}(x)||g_{j}(x)|dx \leq c  \sum_{j=0}^{\infty } K_j\leq c
 \end{equation*}
 where $c$ is independent of $f_j, g_j, j\in \mathbb{N}_{0}$, this proves (\ref{st1-ene})
which finishes the proof of~\textit{Step~1}.\\
Next, we prove the reverse inequality, i.e.,
\begin{equation*}\label{rev-dine}
\big\|(f_{j})_{j\in \mathbb{N}_{0}}\big\|_{\ell ^{q(\cdot
)}(L^{p(\cdot )})}\lesssim \sup_{( g_{j}) _{j\in \mathbb{N
}_{0}} \in U_{p',q'}} \int_{\mathbb{R}^{n}}\sum_{j=0}^{\infty }|f_{j}(x)||g_{j}(x)|dx.
\end{equation*}
\textbf{Step 2:} First, we consider the case where $p^+<\infty$ or $p(\cdot)=\infty$ and $q^+<\infty$. Denote $K :=\|(f_{j})_{j\in \mathbb{N}_{0}} \|_{\ell^{q(\cdot)}(L^{p(\cdot )})}=\|(|f_{j}|)_{j\in \mathbb{N}_{0}} \|_{\ell^{q(\cdot)}(L^{p(\cdot )})} $, we suppose that $f_j,j\in \mathbb{N}_{0}$ are real  positive-valued functions and $K>0$ since when $K=0$ the result is obvious, let  
\begin{align*}
\beta_j:=\inf \{\lambda _{j}>0:\varrho
_{p(\cdot )}(f_{j}/K \lambda _{j}^{1/q(\cdot )})\leq 1\},\qquad j\in \mathbb{N}_{0},
\end{align*}
 if $\beta_j=0$ for some $j\in \mathbb{N}_{0}$ then $f_j=0$, hence we can suppose that  $\beta_j>0, \text{ for every } j\in \mathbb{N}_{0}$, we aim to prove that 
\begin{align}
\sum_{j=0}^{\infty }\beta_j=1 \text{ and } \varrho
_{p(\cdot )}\Big(f_{j}/K \beta_{j}^{1/q(\cdot )}\Big)=1 \text{ for avery } j\in \mathbb{N}_{0}.\label{2eq1}
\end{align}

By the definition of $\|\cdot\|_{\ell ^{q(\cdot)}(L^{p(\cdot )})}$, for every  $\varepsilon,\delta>0 $ there exist $\lambda _{j}>0, j\in \mathbb{N}_{0}$ such that  
\begin{align*}
\sum_{j=0}^{\infty }\lambda _{j}\leq 1+\delta \text{ and  } \varrho
_{p(\cdot )}\Big(\frac{f_{j}}{(K+\varepsilon)\lambda _{j}^{1/q(\cdot )}}\Big)\leq 1,
\end{align*}
for  every  $s_j:=\lambda _{j}[(K+\varepsilon)/K]^{q^+},  j\in \mathbb{N}_{0}$, we have 
\begin{align*}
\varrho_{p(\cdot )}\Big(\frac{f_{j}}{K s_{j}^{1/q(\cdot )}}\Big)\leq 
\varrho
_{p(\cdot )}\Big(\frac{f_{j}}{(K+\varepsilon)\lambda _{j}^{1/q(\cdot )}}\Big)\leq 1
\text{ and }
\sum_{j=0}^{\infty }s_{j}\leq (1+\delta)[(K+\varepsilon)/K]^{q^+}, 
\end{align*}
therefore
\begin{align*}
\sum_{j=0}^{\infty }\inf \Big\{\lambda _{j}>0:\varrho
_{p(\cdot )}\Big(\frac{f_{j}}{K\lambda _{j}^{1/q(\cdot )}}\Big)\leq 1\Big\}\leq \sum_{j=0}^{\infty }s_{j}\leq (1+\delta)[(K+\varepsilon)/K]^{q^+},
\end{align*}
since $\varepsilon,\delta$ are arbitrary we conclude that $\varrho _{\ell ^{q(\cdot )}(L^{p\left( \cdot \right)
})}( (f_{j})_{j\in\mathbb{N}_0}/K)=\sum_{j=0}^{\infty }\beta_j\leq1$. Now, if  
$\varrho _{\ell ^{q(\cdot )}(L^{p\left( \cdot \right)
})}( (f_{j})_{j\in\mathbb{N}_0}/K)<1$ and 
since the function $\mu \to \varrho _{\ell ^{q(\cdot )}(L^{p\left( \cdot \right)
})}((f_{j})_{j\in\mathbb{N}_0}/\mu)$ is continuous on $]0;+\infty[$ by Lemma \ref{cont-lem}, there exists $K' < K$ such that $\varrho _{\ell ^{q(\cdot )}(L^{p\left( \cdot \right)
})}((f_{j})_{j\in\mathbb{N}_0}/K')<1$ which makes a contradiction, this proves that $\sum_{j=0}^{\infty }\beta_j=1$.

 For every $j\in \mathbb{N}_{0}$ and $\lambda _{j}>\beta_j$, we have 
\begin{align*}
\varrho
_{p(\cdot )}\Big(f_{j}/K \beta_{j}^{1/q(\cdot )}\Big)&\leq (\lambda_j/\beta_{j})^{p^+}
\varrho
_{p(\cdot )}\Big(f_{j}/K \lambda_{j}^{1/q(\cdot )}\Big)\leq (\lambda_j/\beta_{j})^{p^+} &,\text{ if } p^+<\infty;\\
\varrho
_{p(\cdot )}\Big(f_{j}/K \beta_{j}^{1/q(\cdot )}\Big)&\leq \lambda_j/\beta_{j}
\varrho
_{p(\cdot )}\Big(f_{j}/K \lambda_{j}^{1/q(\cdot )}\Big)\leq \lambda_j/\beta_{j} &,\text{ if } p(\cdot)=\infty,
\end{align*}
therefore by the definition of $\beta_j$, we have $\varrho
_{p(\cdot )}(f_{j}/K \beta_{j}^{1/q(\cdot )})\leq 1$. Now, if $\varrho
_{p(\cdot )}(f_{j}/K \beta_{j}^{1/q(\cdot )})< 1$ for some $j$,  there exists $0<\beta_j' < \beta_j$ such that $\varrho
_{p(\cdot )}(f_{j}/K \beta_{j}'^{\,1/q(\cdot )})< 1$ which makes a contradiction( $\sum_{j=0}^{\infty }\beta_j\text{ becomes strictly less than } 1$), hence,  $\varrho
_{p(\cdot )}(f_{j}/K \beta_{j}^{1/q(\cdot )})=1$ for every  $j\in \mathbb{N}_{0}$.

If $p^+<\infty$, then  by (\ref{2eq1}) we have,  for every $j\in \mathbb{N}_{0}$,
\begin{align*}
K \beta_{j}=\int_{\mathbb{R}^n} f_j \beta_{j}^{1/q'(\cdot )} \Big(f_{j}/K \beta_{j}^{1/q(\cdot )}\Big)^{p-1} ,
\end{align*}
thus, 
\begin{align}
K=K \sum_{j=0}^{\infty }\beta_j=\sum_{j=0}^{\infty } \, \int_{\mathbb{R}^n} f_j h_j \label{equa existence}
\end{align}
where $h_j:= \beta_{j}^{1/q'(\cdot )} \Big(f_{j}/K \beta_{j}^{1/q(\cdot )}\Big)^{p-1}$, it rests only to prove that $\big\|(h_{j})_{j\in \mathbb{N}_{0}}\big\|_{\ell ^{q'(\cdot)}(L^{p'(\cdot )})}\leq 1$, it suffices to prove that 
$\varrho _{\ell ^{q'(\cdot )}(L^{p'\left( \cdot \right)})}((h_{j})_{j\in\mathbb{N}_0})\leq 1$, this last is correct  since  
\begin{align*}
\varrho _{p'(\cdot )}(h_j/\beta_{j}^{1/q'(\cdot )})&=\int_{\mathbb{R}^{n}}\omega _{p'(x)}(h_j(x)/\beta_{j}^{1/q'(x )})\,dx\\
&\leq \varrho _{p(\cdot )}(f_{j}/K \beta_{j}^{1/q(\cdot )})\\
&\leq 1 
\end{align*}%
and $\sum_{j=0}^{\infty }\beta_j=1$.

If $p(\cdot)=\infty$, let $\varepsilon>0$, for any $j\in \mathbb{N}_{0}$  there exists a bounded set $E_j\subset\mathbb{R}^n$ with $|E_j|>0$ such that $f_{j}(x)/K \beta_{j}^{1/q(x)}+ \varepsilon /K2^{j-1} > \|f_{j}/K \beta_{j}^{1/q(\cdot )}\|_\infty$ for any $x\in E_j$, let $T_j=|E_j|^{-1}\chi_{E_j}$, we have   $T_j \in L^1$ and by (\ref{2eq1}),
\begin{align*}
K\beta_j\leq \int_{\mathbb{R}^n} f_{j} \beta_{j}^{1/q'(\cdot )} T_j + \varepsilon /2^{j-1},
\end{align*}
hence
\begin{align*}
K=K \sum_{j=0}^{\infty }\beta_j\leq\sum_{j=0}^{\infty } \, \int_{\mathbb{R}^n} f_j h_j +\varepsilon,
\end{align*}
where $h_j:= \beta_{j}^{1/q'(\cdot )} T_j$, we can easily see that $\|(h_{j})_{j\in \mathbb{N}_{0}}\|_{\ell ^{q'(\cdot)}(L^{1})}\leq 1$,thus 
\begin{equation*}
\big\|(f_{j})_{j\in \mathbb{N}_{0}}\big\|_{\ell ^{q(\cdot
)}(L^{\infty})}\leq \sup_{( g_{j}) _{j\in \mathbb{N
}_{0}} \in U_{1,q'}} \int_{\mathbb{R}^{n}}\sum_{j=0}^{\infty }f_{j}(x)|g_{j}(x)|dx.
\end{equation*}%

Now, let $p\in \mathcal{P}(\mathbb{R}^{n})$, by Lemma \ref{Norm cong Form} and  (\ref{q=infty norm})  we can see that
\begin{equation*}
\big\|(f_{j})_{j\in \mathbb{N}_{0}}\big\|_{\ell ^{\infty}(L^{p(\cdot)})}\leq \sup_{( g_{j}) _{j\in \mathbb{N
}_{0}} \in U_{p',1}} \int_{\mathbb{R}^{n}}\sum_{j=0}^{\infty }f_{j}(x)|g_{j}(x)|dx,
\end{equation*}%
for every  $(f_{j})_{j\in \mathbb{N}_{0}}\in\ell ^{\infty}(L^{p(\cdot )})$.

Now, let $p,q\in \mathcal{P}(\mathbb{R}^{n})$ with \textit{real values} (i.e., $p(x),q(x)\in [1,\infty)$ for a.e. $x\in \mathbb{R}^n$),  $\left( f_{j}\right) _{j\in \mathbb{N}_{0}}\in\ell ^{q(\cdot)}(L^{p(\cdot )})$, for $(n, j) \in \mathbb{N}\times\mathbb{N}_{0}$ define
$A_n:=\{x\in \mathbb{R}^n | p(x)\leq n, q(x)\leq n \}$ and $f_j^n=\chi_{A_n}f_j$, it's clear that
$f_{j}^n\leq f_{j}^{n+1}$ for every $(n, j) \in \mathbb{N}\times\mathbb{N}_{0}$ and $\lim_n f_j^n(x) =f_j(x)$ for all $j\in \mathbb{N}_0, x \in \mathbb{R}^n$, by Lemma \ref{Thm inc} we have,
\begin{equation}
\big\|(f_{j})_{j\in \mathbb{N}_{0}}\big\|_{\ell ^{q(\cdot
)}(L^{p(\cdot )})}=\sup_n \big\|(f_{j}^n)_{j\in \mathbb{N}_{0}}\big\|_{\ell ^{q(\cdot
)}(L^{p(\cdot )})},
\end{equation}
let $ \tilde{p}_n:=\chi_{A_n} p + n \chi_{\mathbb{R}^n\backslash A_n} $ and $\tilde{q}_n:=\chi_{A_n} q + n \chi_{\mathbb{R}^n\backslash A_n}$, since $\tilde{p}_n^+<\infty$ and $\tilde{q}_n^+<\infty$, by (\ref{equa existence}),   for every $n \in \mathbb{N}$ there exists  $(
h_{j}^n) _{j\in \mathbb{N}_{0}}\in U_{\tilde{p}_n',\tilde{q}_n'}$ such that
\begin{align*}
 \big\|(f_{j}^n)_{j\in \mathbb{N}_{0}}\big\|_{\ell ^{\tilde{q}_n(\cdot
)}(L^{\tilde{p}_n(\cdot )})}=\int_{\mathbb{R}^{n}}\sum_{j=0}^{\infty }f_{j}^nh_{j}^n,
\end{align*}
hence
\begin{align*}
\sup_n \big\|(f_{j}^n)_{j\in \mathbb{N}_{0}}\big\|_{\ell ^{q(\cdot
)}(L^{p(\cdot )})}&=
\sup_n \big\|(f_{j}^n)_{j\in \mathbb{N}_{0}}\big\|_{\ell ^{\tilde{q}_n(\cdot
)}(L^{\tilde{p}_n(\cdot )})}\\
& = 
\sup_n  \int_{\mathbb{R}^{n}}\sum_{j=0}^{\infty }f_{j}^nh_{j}^n\\
&\leq \sup_{(g_{j})_{j\in \mathbb{N}_{0}} \in U_{p',q'}} \int_{\mathbb{R}^{n}}\sum_{j=0}^{\infty }f_j(x)|g_{j}(x)|dx,
\end{align*}
 therefore 
\begin{align*}
\big\|(f_{j})_{j\in \mathbb{N}_{0}}\big\|_{\ell ^{q(\cdot
)}(L^{p(\cdot )})}\leq \sup_{( g_{j}) _{j\in \mathbb{N
}_{0}} \in U_{p',q'}} \int_{\mathbb{R}^{n}}\sum_{j=0}^{\infty }f_{j}(x)|g_{j}(x)|dx.
\end{align*}
Similarly, we get 
\begin{align*}
\big\|(f_{j})_{j\in \mathbb{N}_{0}}\big\|_{\ell ^{q(\cdot
)}(L^\infty)}\leq \sup_{( g_{j}) _{j\in \mathbb{N
}_{0}} \in U_{1,q'}} \int_{\mathbb{R}^{n}}\sum_{j=0}^{\infty }f_{j}(x)|g_{j}(x)|dx,
\end{align*}
for every   $\left( f_{j}\right) _{j\in \mathbb{N}_{0}}\in\ell ^{q(\cdot)}(L^\infty)$ where $q\in \mathcal{P}(\mathbb{R}^{n})$ is of \textit{real values}(i.e., $q(x)\in [1,\infty)$ for a.e. $x\in \mathbb{R}^n$).
\\
\textbf{Step 3:} Let $p,q\in \mathcal{P}(\mathbb{R}^{n})$, $\left( f_{j}\right) _{j\in \mathbb{N%
}_{0}}\in\ell ^{q(\cdot)}(L^{p(\cdot )})$ , $A:=\{x\in \mathbb{R}^n | p(x)<\infty \},B:=\{x\in \mathbb{R}^n | q(x)<\infty \}$, for every $j\in \mathbb{N}_{0}$, $f_j=\chi_{A\cap B}f_j+\chi_{\mathbb{R}^n\backslash B}f_j+\chi_{B\backslash A}f_j$, by the arguments of \textit{Step 2}, we have
\begin{align*}
\big\|(f_{j})_{j\in \mathbb{N}_{0}}\big\|_{\ell ^{q(\cdot
)}(L^{p(\cdot )})}&\lesssim \big\|(\chi_{A\cap B}f_j)_{j\in \mathbb{N}_{0}}\big\|_{\ell ^{q(\cdot
)}(L^{p(\cdot )})}
+ 
\big\|(\chi_{\mathbb{R}^n\backslash B}f_j)_{j\in \mathbb{N}_{0}}\big\|_{\ell ^{\infty}(L^{p(\cdot )})}\\
&\qquad +
\big\|(\chi_{B\backslash A}f_j)_{j\in \mathbb{N}_{0}}\big\|_{\ell ^{q(\cdot
)}(L^{\infty})}\\
&\lesssim \sup_{( g_{j}) _{j\in \mathbb{N
}_{0}} \in U_{p',q'}} \int_{\mathbb{R}^{n}}\sum_{j=0}^{\infty }f_{j}(x)|g_{j}(x)|dx.
\end{align*}
 The proof is complete. 
\end{proof}

The corresponding result for the space $\ell ^{q}(L^{p(\cdot )})$ were presented in \cite[Lemma 6.1]{HN18}
where $q\in [1,\infty]$ is  constant, $p^->1$ and $p^+<\infty$. In \cite[Proposition 1]{Dua1}, the following inequality was proven, 
\begin{equation*}
\sum_{j=0}^{\infty }\int_{\mathbb{R}^n} |f_{j}(x)||g_{j}(x)|dx \leq c
\big\|(f_{j})_{j\in \mathbb{N}_{0}}\big\|_{\ell ^{q(\cdot
)}(L^{p(\cdot )})}
\big\|(g_{j})_{j\in \mathbb{N}_{0}}\big\|_{\ell ^{q'(\cdot
)}(L^{p'(\cdot )})}
\end{equation*}
for every 
$(f_{j})_{j\in \mathbb{N}_{0}}\in \ell ^{q(\cdot
)}(L^{p(\cdot )})$ and 
$(g_{j})_{j\in \mathbb{N}_{0}}\in \ell ^{q'(\cdot)}(L^{p'(\cdot )})$ of locally Lebesgue integrable functions, where 
$c=2(1+1/p^- -1/p^+)$,  $1<p^-<p^+<\infty$ and $1<q^-<q^+<\infty$. In Step 1 of Lemma \ref{duality2},  we get this inequality for arbitrary $p,q\in \mathcal{P}(\mathbb{R}^{n})$ by employing a distinct method than that used in \cite{Dua1}. 

 We finish this section by generalizing  H\"{o}lder’s inequality. By \cite[Lemma 3.2.20]{DHHR}, if $ p,q,s\in \mathcal{P}(\mathbb{R}^{n})$ are such that $1/s=1/p+1/q$, then  for every $f\in L^{p(\cdot)}$ and $g\in L^{p(\cdot)}$
\begin{equation}
\| fg\big\|_{s(\cdot)}\lesssim \| f\big\|_{p(\cdot)}\| g\big\|_{q(\cdot)}. \label{hold-in}
\end{equation} 
Similarly,  for exponents of $\mathcal{P}(\mathbb{R}^{n})$, if $1/p=1/p_1+1/p_2$ and $1/q=1/q_1+1/q_2$, then 
\begin{align}
\Vert \left( f_{j}g_j\right) _{j\in \mathbb{N}_{0}}\Vert _{\ell ^{q(\cdot
)}(L^{p\left( \cdot \right) })} \lesssim 
\Vert \left( f_{j}\right) _{j\in \mathbb{N}_{0}}\Vert _{\ell ^{q_1(\cdot
)}(L^{p_1\left( \cdot \right) })} 
\Vert \left( g_{j}\right) _{j\in \mathbb{N}_{0}}\Vert _{\ell ^{q_2(\cdot
)}(L^{p_2\left( \cdot \right) })},\label{Hold-GEN}
\end{align}
in particular if $1/p=1/p_1+1/p_2$, then 
\begin{align}
\Vert \left( f_{j}g_j\right) _{j\in \mathbb{N}_{0}}\Vert _{\ell ^{q(\cdot
)}(L^{p\left( \cdot \right) })} \lesssim 
\sup_{k\in \mathbb{N}_{0}}\big\| f_k \big\|_{p_1(\cdot)} 
\Vert \left( g_{j}\right) _{j\in \mathbb{N}_{0}}\Vert _{\ell ^{q(\cdot
)}(L^{p_2\left( \cdot \right) })}, \label{Hold-inq}
\end{align}
the proof  follows standard techniques similar to that used above 
with the aid of (\ref{hold-in}).
\section{\textbf{The Results And Their Proofs}}
\label{Mainsec}
In this section we present and prove the estimates for the commutators $\lbrack V\cdot \nabla ,\Delta _{j}]f$,
we follow  the approachs outlined in \cite{art1,HN18} and make use of some  techniques presented therein.

\subsection{\textbf{Preliminary Lemmas}}

The next lemma is a Hardy-type inequality, see  \cite[Chapter~4]{Phd-ths} for a more general statement. 

\begin{lemma}\label{thm compar}
Let $p,q\in \mathcal{P}(\mathbb{R}^{n})$, $0<a<1$  and $\{g_m\}_{m \in \mathbb{N}_0}\in \ell ^{q(\cdot )}(L^{p\left( \cdot \right) }) $ . For every $j \in \mathbb{N}_0$ and $x \in \mathbb{R}^n$, let $G_j(x)=\sum_{m \geq j } a^{m-j}g_m(x)$ and $H_j(x)=\sum_{m \leq j } a^{j-m}g_m(x)$. Then
\begin{align*}
\left\Vert (G_j)_{j \in \mathbb{N}_0}\right\Vert _{\ell ^{q(\cdot )}(L^{p\left( \cdot \right) })}
+ \left\Vert (H_j)_{j \in \mathbb{N}_0}\right\Vert _{\ell ^{q(\cdot )}(L^{p\left( \cdot \right) })}
\lesssim \left\Vert (g_m)_{m \in \mathbb{N}_0}\right\Vert _{\ell ^{q(\cdot )}(L^{p\left( \cdot \right) })}.
\end{align*}
\end{lemma}
\begin{proof}
By scaling arguments it suffices to suppose that  $\left\Vert (g_m)_{m \in \mathbb{N}_0}\right\Vert _{\ell ^{q(\cdot )}(L^{p\left( \cdot \right) })}\leq 1$, which is equivalent to  $\varrho _{\ell ^{q(\cdot )}(L^{p\left( \cdot \right)
})}((g_m)_{m \in \mathbb{N}_0})\leq 1$,  by definition,  for any $\varepsilon>0$ there exist $\lambda_m>0,m \in \mathbb{N}_0 $ such that 
\begin{equation*}
\sum_{m=0}^{\infty }\lambda _{m}\leq 1+\varepsilon \text{ and  } \varrho
_{p(\cdot )}\Big(\frac{g_{m}}{\lambda _{m}^{1/q(\cdot )}}\Big)\leq 1,
\end{equation*}
let $0<\gamma<q^-$, $\beta_j:=(1-a^{ \gamma}) \sum_{m\geq j} a^{(m-j)\gamma} \lambda_m$ and $ \theta_j:=
(1-a^{ \gamma}) \sum_{0\leq m\leq j} a^{(j-m)\gamma} \lambda_m$, we have 
\begin{footnotesize}
\begin{align*}
 \sum_{j=0}^{\infty }  \sum_{m\geq j} a^{(m-j)\gamma} \lambda_m &=\sum_{j=0}^{\infty } \sum_{k=0}^{\infty } a^{ k \gamma} \lambda_{j+k} \qquad & 
 \sum_{j=0}^{\infty }  \sum_{0\leq m\leq j} a^{(j-m)\gamma} \lambda_m  &=\sum_{j=0}^{\infty } \sum_{0\leq k\leq j}^{\infty } a^{ k \gamma} \lambda_{j-k} 
 \\
 &=\sum_{k=0}^{\infty } a^{ k \gamma} \sum_{j=0}^{\infty }  \lambda_{j+k} 
& &=\sum_{k=0}^{\infty } a^{ k \gamma} \sum_{j=k}^{\infty }  \lambda_{j-k}
\\
&\leq \frac{1+\varepsilon}{1-a^{ \gamma}},
& &\leq \frac{1+\varepsilon}{1-a^{ \gamma}}, 
\end{align*}
\end{footnotesize}
and 
\begin{footnotesize}
\begin{align*}
\frac{G_j(x)}{\beta_j^{1/q(x)}}&=\sum_{m \geq j } a^{m-j} (\lambda_m/\beta_j)^{1/q(x)} \frac{g_m(x)}{\lambda_m^{1/q(x)}} &
 \frac{H_j(x)}{\theta_j^{1/q(x)}}&=\sum_{0\leq m\leq j } a^{j-m} (\lambda_m/\theta_j)^{1/q(x)} \frac{g_m(x)}{\lambda_m^{1/q(x)}}
\\
&\leq  \frac{1}{(1-a^{ \gamma})^\frac{1}{q^-}}\sum_{m \geq j } a^{(m-j)(1-\gamma/q^-)} \frac{g_m(x)}{\lambda_m^{1/q(x)}}\, ,
& & \leq \frac{1}{(1-a^{ \gamma})^\frac{1}{q^-}}\sum_{0\leq m\leq j } a^{(j-m)(1-\gamma/q^-)} \frac{g_m(x)}{\lambda_m^{1/q(x)}}\,,
\end{align*}
\end{footnotesize}
therefore  $\sum_{j=0}^{\infty }\beta_j\leq 1+\varepsilon$, $\sum_{j=0}^{\infty }\theta_j\leq 1+\varepsilon$, $\varrho _{p(\cdot )}(c G_j/\beta_j^{1/q(\cdot)})\leq 1$ and 
$\varrho _{p(\cdot )}(c H_j/\theta_j^{1/q(\cdot)})\leq 1$  
  with $c=(1-a^{ \gamma})^{1/q^-}(1-a^{1-\gamma/q^-})$, these implies that $\varrho _{\ell ^{q(\cdot )}(L^{p\left( \cdot \right)})}((cG_j)_{j \in \mathbb{N}_0})\leq 1+\varepsilon$ and 
 $\varrho _{\ell ^{q(\cdot )}(L^{p\left( \cdot \right)})}((cH_j)_{j \in \mathbb{N}_0})\leq 1+\varepsilon$. By letting $\varepsilon$ go to zero we conclude that 
$\left\Vert (cG_j)_{j \in \mathbb{N}_0}\right\Vert _{\ell ^{q(\cdot )}(L^{p\left( \cdot \right) })}\leq 1$ and 
$\left\Vert (cH_j)_{j \in \mathbb{N}_0}\right\Vert _{\ell ^{q(\cdot )}(L^{p\left( \cdot \right) })}\leq 1$, which completes the proof.
\end{proof}
Let $(\mathcal{F}\varphi _{j})_{j\in \mathbb{N}_{0}}$ be a smooth dyadic
resolution of unity.\ Let $\Psi \in \mathcal{S}(\mathbb{R}^{n})$ and\ 
\begin{equation*}
\Lambda _{j,m}(f,g)(x):=\int_{\mathbb{R}^{2n}}\varphi _{j}(x-y)(\Psi
_{m}(x-z)-\Psi _{m}(y-z))f(y)\varphi _{m}\ast g(z)dydz,
\end{equation*}
where $j,m\in \mathbb{N}_{0}$ and $\Psi _{m}=2^{m}\Psi (2^{m}\cdot )$.
\begin{lemma}
\label{HN1}Let $s\in C_{\mathrm{loc}}^{\log }(\mathbb{R}^{n}),a\in \mathbb{R}%
,p,p_{1},p_{2}\in \mathcal{P}^{\log }(\mathbb{R}^{n})$, $%
q\in \mathcal{P}\left( \mathbb{R}^{n}\right) $ with $1/q\in C_{\mathrm{loc}}^{\log } (\mathbb{R}^{n})$, such that $1/p=1/p_1+1/p_2$ and $(s+a)^{-}>0$. Then%
\begin{equation*}
\sum_{j=0}^{\infty }\sum_{m=j}^{\infty }\int_{\mathbb{R}^{n}}2^{ja}\big|\Lambda
_{j,m}(f,g)(x) h_{j}(x)\big|\,dx\leq c\big\|f\big\|_{p_{1}(\cdot )}\big\|g\big\|%
_{B_{p_{2}(\cdot ),q(\cdot )}^{s(\cdot )+a}}
\end{equation*}%
holds for any sequence $(h_{j})_{j\in \mathbb{N}_{0}}$ of measurable functions that satisfies
\begin{equation}
\big\|(2^{-js(\cdot )}h_{j})_{j\in \mathbb{N}_{0}}\big\|_{\ell ^{q'(\cdot
)}(L^{p'(\cdot )})}\leq 1.\label{fun-h1}
\end{equation}
\end{lemma}
\begin{proof} Since 
$\varphi ,\Psi \in \mathcal{S}(\mathbb{R}^{n})$, we have%
\begin{equation*}
|\varphi _{j}|\leq c\eta _{j,N}\quad \text{and}\quad |\Psi _{m}|\leq c\eta
_{m,N},\quad j,m\in \mathbb{N}_{0},N>n,
\end{equation*}%
where $c$  is independent of $j$ and $m$ and  $N$ can be selected sufficiently large, therefore 
\begin{align*}
|\Lambda _{j,m}(f,g)(x)| &\lesssim  \int_{\mathbb{R}^{2n}}\eta _{j,N}(x-y)\eta _{m,N}(x-z)|f(y)||\varphi
_{m}\ast g(z)|dydz \\
&\quad +\int_{\mathbb{R}^{2n}}\eta _{j,N}(x-y)\eta _{m,N}(y-z)|f(y)||\varphi
_{m}\ast g(z)|dydz \\
&= c H_{j,m}(x)+c I_{j,m}(x)
\end{align*}%
where
\begin{align}
H_{j,m}(x)&:= (\eta _{j,N}\ast |f|(x))(\eta _{m,N}\ast |\varphi _{m}\ast
g|(x)); \label{firsttrm} \\
I_{j,m}(x)&:=\eta _{j,N}\ast (|f|\eta _{m,N}\ast |\varphi _{m}\ast
g|)(x),  \label{second}
\end{align}
for all $x\in \mathbb{R}^{n},j,m\in \mathbb{N}_{0}.$
We begin by estimating the term (\ref{firsttrm}),
we have $s\in C_{\mathrm{loc}}^{\log }(\mathbb{R}^{n})$ then by Lemma \ref{DHR-lemma},
\begin{equation}
\sum_{j=0}^{\infty }\sum_{m=j}^{\infty }\int_{\mathbb{R}%
^{n}}2^{ja}H_{j,m}(x)|h_{j}(x)|dx \lesssim 
\int_{\mathbb{R}^{n}}\sum_{j=0}^{\infty
} \eta _{j,N}\ast |f|(x) \vartheta_j(x)
2^{-js(x)}|h_{j}(x)|dx
 \label{est1}
\end{equation}%
where 
\begin{align*}
\vartheta_j(x):= \sum_{m=j}^{\infty }2^{(j-m)(a+s)^-}  \eta _{m,N_{1}}\ast (2^{m(a+s(\cdot
))}|\varphi _{m}\ast g|)(x),\, x\in \mathbb{R}^{n},j\in \mathbb{N}_{0}
\end{align*}
and $N_1>n$ is sufficiently  large, by Lemmas \ref{thm compar} and  \ref{Alm-Hastolemma1} we have
\begin{align*}
\big\|(\vartheta_j
)_{j\in \mathbb{N}_{0}}\big\|_{\ell ^{q(\cdot
)}(L^{p_2(\cdot )})} \lesssim 
\big\|(2^{j(a+s(\cdot
))}|\varphi _{j}\ast g|
)_{j\in \mathbb{N}_{0}}\big\|_{\ell ^{q(\cdot
)}(L^{p_2(\cdot )})}, 
\end{align*}
therefore by  Lemma \ref{duality2} and inequalities (\ref{fun-h1}), (\ref{eneq p norm}) and (\ref{Hold-inq}), 
\begin{align*}
\sum_{j=0}^{\infty }\sum_{m=j}^{\infty }\int_{\mathbb{R}%
^{n}}2^{ja}H_{j,m}(x)|h_{j}(x)|dx&\lesssim
\big\|(\eta _{j,N}\ast |f|(\cdot)\vartheta_j(\cdot))_{j\in \mathbb{N}_{0}}\big\|_{\ell ^{q(\cdot
)}(L^{p(\cdot )})} \\
&\lesssim \sup_{k} \big\|\eta _{k,N}\ast |f|\big\|_{p_{1}(\cdot )}
\big\|( \vartheta_j
)_{j\in \mathbb{N}_{0}}\big\|_{\ell ^{q(\cdot
)}(L^{p_2(\cdot )})}\\
&\lesssim \big\|f\big\|_{p_{1}(\cdot )}\big\|g\big\|%
_{B_{p_{2}(\cdot ),q(\cdot )}^{s(\cdot )+a}}.
\end{align*}

Now, with similar arguments we estimate the term $\mathrm{
\eqref{second}}$ as follows:
\begin{align*}
\sum_{j=0}^{\infty }\sum_{m=j}^{\infty }\int_{\mathbb{R}%
^{n}}2^{ja}I_{j,m}(x)|h_{j}(x)|dx 
\lesssim \int_{\mathbb{R}^{n}}\sum_{j=0}^{\infty
}2^{-js(x)}|h_{j}(x)|\eta _{j,N_2}\ast (|f| \kappa _{j})(x)dx 
\end{align*}
where 
\begin{equation*}
\kappa _{j}(x):= \sum_{m=j}^{\infty } 2^{(j-m)(a+s)^-} \eta _{m,N_3}\ast |2^{m(a+s(\cdot))} \varphi _{m}\ast
g|(x),\, x\in \mathbb{R}^{n},j\in \mathbb{N}_{0},
\end{equation*}
with $N_2,N_3>n$ sufficiently  large. Again, Lemma \ref{duality2}, inequalities (\ref{fun-h1}), (\ref{Hold-inq}) and Lemmas  \ref{thm compar}, \ref{Alm-Hastolemma1} yield
\begin{align*}
\sum_{j=0}^{\infty }\sum_{m=j}^{\infty }\int_{\mathbb{R}%
^{n}}2^{ja}I_{j,m}(x)|h_{j}(x)|dx  &\lesssim  \big\| (|f| \kappa _{j})_{j\in \mathbb{N}_{0}}\big\|_{\ell ^{q(\cdot
)}(L^{p(\cdot )})}\\
&\lesssim \big\|f\big\|_{p_{1}(\cdot )}
\big\|(\kappa _{j}
)_{j\in \mathbb{N}_{0}}\big\|_{\ell ^{q(\cdot
)}(L^{p_2(\cdot )})}
\\
&\lesssim \big\|f\big\|_{p_{1}(\cdot )}
\big\|(\eta _{m,N_3}\ast |2^{m(a+s(\cdot))} \varphi _{m}\ast
g|
)_{m\in \mathbb{N}_{0}}\big\|_{\ell ^{q(\cdot
)}(L^{p_2(\cdot )})}\\
&\lesssim \big\|f\big\|_{p_{1}(\cdot )}\big\|g\big\|%
_{B_{p_{2}(\cdot ),q(\cdot )}^{s(\cdot )+a}},
\end{align*}%
which completes the proof.
\end{proof}

For $0\leq m\leq j,j,m\in \mathbb{N}_{0}$, $x\in \mathbb{R}^{n}$ and $K\in 
\mathbb{N}$, we set%
\begin{align*}
E_{j,m,K}(f,g)(x)&=2^{(m-j)K}\int_{\mathbb{R}^{2n}}\eta _{j,N}(x-y)\eta
_{m,N}(x-z)|f(y)||\varphi _{m}\ast g(z)|dydz \\
&\quad +2^{(m-j)K}\int_{\mathbb{R}^{2n}}\eta _{j,N}(x-y)\eta
_{m,N}(y-z)|f(y)||\varphi _{m}\ast g(z)|dydz,
\end{align*}%
where $N>n$ is large enough.
\begin{lemma}
\label{HN2}Let $s\in C_{\mathrm{loc}}^{\log }(\mathbb{R}^{n}),a\in \mathbb{R},K\in \mathbb{N}%
,p,p_{1},p_{2}\in \mathcal{P}^{\log }(\mathbb{R}^{n})$ and $%
q\in \mathcal{P}(\mathbb{R}^{n}) $ with $1/q\in C_{\mathrm{loc}}^{\log } (\mathbb{R}^{n})$. Assume that $1/p=1/p_1+1/p_2$ and $(s+a)^{+}<K$. Then%
\begin{equation*}
\sum_{j=0}^{\infty }\sum_{m=0}^{j}\int_{\mathbb{R}%
^{n}}2^{ja}E_{j,m,K}(f,g)(x)|h_{j}(x)|dx\lesssim \big\|f\big\|_{p_{1}(\cdot
)}\big\|g\big\|_{B_{p_{2}(\cdot ),q(\cdot )}^{s(\cdot )+a}}
\end{equation*}%
holds  for any sequence $(h_{j})_{j\in \mathbb{N}_{0}}$ of measurable functions that satisfies
\begin{equation}
\big\|(2^{-js(\cdot )}h_{j})_{j\in \mathbb{N}_{0}}\big\|_{\ell ^{q'(\cdot
)}(L^{p'(\cdot )})}\leq 1.\label{fun-h}
\end{equation}
\end{lemma}
\begin{proof}  Let $0\leq m\leq j,j,m\in \mathbb{N}_{0}$, for every $x\in \mathbb{R}^{n}$ we have 
\begin{eqnarray*}
2^{ja}E_{j,m,K}(f,g)(x) &\lesssim &2^{(m-j)K+aj}\big(\eta _{j,N}\ast
|f|(x)\eta _{m,N}\ast |\varphi _{m}\ast g|(x)+I_{j,m}(x)\big) \\
&=&2^{(m-j)K+aj}\big(H_{j,m}(x)+I_{j,m}(x)\big),
\end{eqnarray*}%
where $H_{j,m}$ and $I_{j,m}$ are defined in (\ref{firsttrm}) and  (\ref{second}) respectively, and $N$ is large enough.  Lemma \ref{DHR-lemma}  yields
\begin{equation*}
\int_{\mathbb{R}^{n}} \sum_{j=0}^{\infty }\sum_{m=0}^{j}2^{(m-j)K+aj}H_{j,m}(x) |h_j(x)|\,dx
\lesssim \int_{\mathbb{R}^{n}} \sum_{j=0}^{\infty } \eta _{j,N}\ast|f|(x)\vartheta _{j}(x) 2^{-js(x)}|h_j(x)|
\,dx,
\end{equation*}%
where
\begin{eqnarray*}
\vartheta _{j}(x):=\sum_{m=0}^{j}2^{(m-j)(K-(s+a)^+)}\eta _{m,N_{1}}\ast
(2^{m(s(\cdot )+a)}|\varphi _{m}\ast g|)(x),\, x\in \mathbb{R}^{n},j\in \mathbb{N}_{0},
\end{eqnarray*}
and $N_1>n$ is sufficiently large, by Lemma \ref{duality2}, inequalities (\ref{fun-h}) , (\ref{Hold-inq}) and Lemmas  \ref{thm compar}, \ref{Alm-Hastolemma1}, we have
\begin{align*}
\int_{\mathbb{R}^{n}} \sum_{j=0}^{\infty }\sum_{m=0}^{j}2^{(m-j)K+aj}H_{j,m}(x) |h_j(x)|\,dx \lesssim \big\|f\big\|_{p_{1}(\cdot )}\big\|g\big\|%
_{B_{p_{2}(\cdot ),q(\cdot )}^{s(\cdot )+a}}.
\end{align*}

With similar arguments we prove the following estimate:
\begin{align*}
\int_{\mathbb{R}^{n}} \sum_{j=0}^{\infty }\sum_{m=0}^{j}2^{(m-j)K+aj}I_{j,m}(x) |h_j(x)|\,dx
\lesssim \int_{\mathbb{R}^{n}} \sum_{j=0}^{\infty} \eta _{j,N_2}\ast( |f| \kappa_j)(x) 2^{-js(x)}|h_j(x)|  \,dx
\end{align*}
where
\begin{align*}
\kappa_j(x):=\sum_{m=0}^{j}  2^{(m-j)(K-(s+a)^+)}\eta _{m,N_3}  \ast\big( 2^{m(s(\cdot )+a)} |\varphi _{m}\ast
g|\big)(x), \, x\in \mathbb{R}^{n},j\in \mathbb{N}_{0},
\end{align*}
with $N_2,N_3>n$. Again,  Lemma \ref{duality2}, inequalities (\ref{fun-h1}) ,(\ref{Hold-inq}) and Lemmas \ref{thm compar}, \ref{Alm-Hastolemma1} yield
\begin{align*}
\int_{\mathbb{R}^{n}} \sum_{j=0}^{\infty }\sum_{m=0}^{j}2^{(m-j)K+aj}I_{j,m}(x) |h_j(x)|\,dx \lesssim \big\|f\big\|_{p_{1}(\cdot )}\big\|g\big\|%
_{B_{p_{2}(\cdot ),q(\cdot )}^{s(\cdot )+a}},
\end{align*}%
which completes the proof.
\end{proof}
\subsection{\textbf{Main Results}}
\begin{theorem}
\label{result1}Let $s\in C_{\mathrm{loc}}^{\log }(\mathbb{R}^{n}),s^{-}>0,p,p_{1},p_{2}\in 
\mathcal{P}^{\log }(\mathbb{R}^{n})$, $q\in \mathcal{P} (\mathbb{R}^{n})$ with $1/q\in C_{\mathrm{loc}}^{\log } (\mathbb{R}^{n})$ and
$1/p=1/p_1+1/p_2$. Let $V=(V_{1},...,V_{n})\in \left( \mathcal{S}(\mathbb{R}^{n})\right) ^{n}$
be a vector field. Then for any $f\in \mathcal{S}(\mathbb{R}^{n})$%
\begin{align}
\big\|( 2^{js(\cdot )}[V\cdot \nabla ,\Delta _{j}]f )_{j\in \mathbb{N}_{0}}\big\|_{\ell ^{q(\cdot)}(L^{p(\cdot )})} \lesssim 
\big\|\nabla f%
\big\|_{p_{1}(\cdot )}\big\|V\big\|_{B_{p_{2}(\cdot ),q(\cdot )}^{s(\cdot
)}}+A  \label{Est-th1.1}
\end{align}
where%
\begin{equation*}
A=\big\|\nabla V\big\|_{p_{1}(\cdot )}\big\|f\big\|_{B_{p_{2}(\cdot
),q(\cdot )}^{s(\cdot )}}\quad \text{or}\quad A=\big\|V\big\|_{p_{1}(\cdot )}%
\big\|\nabla f\big\|_{B_{p_{2}(\cdot ),q(\cdot )}^{s(\cdot )}}.
\end{equation*}
And 
\begin{align}
\big\|( 2^{js(\cdot )}[V\cdot \nabla ,\Delta _{j}]f )_{j\in \mathbb{N}_{0}}\big\|_{\ell ^{q(\cdot)}(L^{p(\cdot )})} &\lesssim 
\big\|f\mathrm{div}(V)\big\|_{B_{p(\cdot ),q(\cdot )}^{s(\cdot )}}+%
\big\|\nabla V\big\|_{p_{1}(\cdot )}\big\|f\big\|_{B_{p_{2}(\cdot ),q(\cdot
)}^{s(\cdot )}}\notag \\ 
&\hspace{4cm}+\big\|f\big\|_{p_{1}(\cdot )}\big\|V\big\|_{B_{p_{2}(\cdot
),q(\cdot )}^{s(\cdot )+1}}.  \label{Est-th1.1.1}
\end{align}
\end{theorem}
\begin{proof}\qquad\\
\textbf{Step 1.} \textit{Preparation.}
 Let $V=(V_{1},...,V_{n})\in \left(\mathcal{S}(\mathbb{R}^{n})\right) ^{n}$ and $f\in \mathcal{S}(\mathbb{R}%
^{n})$.
For every $x\in \mathbb{R}^n$ we have
\begin{eqnarray*}
\lbrack V\cdot \nabla ,\Delta _{j}]f(x) &=&\sum_{k=1}^{n}V_{k}(x)\partial
_{k}\Delta _{j}f(x)-\Delta _{j}(V_{k}\partial _{k}f)(x) \\
&=&\sum_{k=1}^{n}\int_{\mathbb{R}^{n}}\varphi
_{j}(x-y)(V_{k}(x)-V_{k}(y))\partial _{k}f(y)dy.
\end{eqnarray*}%
Let $(\mathcal{F}\varphi _{j})_{j\in \mathbb{N}_{0}}$\ be a smooth dyadic
resolution of unity. Then
there exist $\Psi_0,\Psi \in \mathcal{S}(\mathbb{R}^{n})$ such that, for all $\xi\in \mathbb{R}^n$,
\begin{equation*}
(\mathcal{F}\phi_0)(\xi)(\mathcal{F}\Psi_0)(\xi)+\sum_{j\in \mathbb{N}} (\mathcal{F}\phi) (2^{-j}\xi) (\mathcal{F}\Psi)(2^{-j}\xi)=1,
\end{equation*} 
therefore $V=\sum_{m\in \mathbb{N}_0}\Psi_m*\phi_m*V$. It follows, for every $x\in \mathbb{R}^{n}$ and every $j\in \mathbb{N}_{0}$,
\begin{align*}
\lbrack V\cdot \nabla ,\Delta _{j}]f(x) &=\sum_{m=0}^{\infty
}\sum_{k=1}^{n}\int_{\mathbb{R}^{2n}}\varphi _{j}(x-y)(\Psi_m(x-z)-\Psi_m(y-z))\partial _{k}f(y)\varphi_{m} \ast
V_{k} (z)\,dz dy \\
&=\sum_{m=0}^{\infty }\sum_{k=1}^{n}\Pi _{j,m,k}(\partial _{k}f,V_{k})(x) \\
&=\sum_{m=0}^{j}\cdot \cdot \cdot +\sum_{m=j+1}^{\infty }\cdot \cdot \cdot,
\end{align*}
to estimate $\lbrack V\cdot \nabla ,\Delta _{j}]f$ in $\ell ^{q(\cdot)}(L^{p(\cdot )})$-norm  we need only to estimate 
\begin{equation}
\left( \sum_{m=0}^{j}\sum_{k=1}^{n}\Pi _{j,m,k}(\partial _{k}f,V_{k})\right)%
_{j\in \mathbb{N}_{0}}\quad \text{and}\quad \left(\sum_{m=j+1}^{\infty
}\sum_{k=1}^{n}\Pi _{j,m,k}(\partial _{k}f,V_{k})\right)_{j\in \mathbb{N}_{0}}
\label{est-th1}
\end{equation}%
in $\ell^{q(\cdot )}( L^{p(\cdot )})$-norm. From Lemma \ref{duality2} we need to estimate%
\begin{equation*}
\int_{\mathbb{R}^{n}}\sum_{j=0}^{\infty }\big|[V\cdot \nabla ,\Delta
_{j}]f(x)h_{j}(x)\big|dx
\end{equation*}%
for every sequence $(h_{j})_{j\in \mathbb{N}_{0}}$ of measurable functions that satisfies 
\begin{equation}
\big\|(2^{-js(\cdot )}h_{j})_{j\in \mathbb{N}_{0}}\big\|_{\ell ^{q'(\cdot
)}(L^{p'(\cdot )})}\leq 1.\label{duh}
\end{equation}
From \cite[Lemma
3.1]{HN18} we derive%
\begin{eqnarray}
\Pi _{j,m,k}(\partial _{k}f,V_{k}) &=&\sum_{1\leq |\alpha |<K}2^{|\alpha
|(m-j)}(\theta _{j,\alpha }\ast \partial _{k}f)(\partial ^{\alpha }\varphi
)_{m}\ast \varphi _{m}\ast V_{k}+\Theta _{j,m,K,k}(\partial _{k}f,V_{k}), 
\notag \\
&=&\sum_{1\leq |\alpha |<K}I_{j,m,\alpha,k}+\Theta
_{j,m,K,k}(\partial _{k}f,V_{k}),  \label{sum1}
\end{eqnarray}%
where%
\begin{align*}
\Theta _{j,m,K,k}(\partial _{k}f,V_{k})(x)=\int_{\mathbb{R}^{2n}}\varphi
_{j}(x-y)\big(\sum_{|\alpha |=K}\frac{1}{\alpha !}(\partial ^{\alpha
}\varphi _{m})(\xi _{\alpha })(y-x)^{\alpha }\big)\partial _{k}f(y)\varphi
_{m}\ast V_{k}(z)dydz,
\end{align*}%
$\xi _{\alpha }$ is on the line segment joining $y-z$ and $x-z$ and%
\begin{equation*}
\theta _{j,\alpha }(x)=\frac{(-1)^{|\alpha |}}{\alpha !}(2^{j}x)^{\alpha
}\varphi _{j}(x),\quad x\in \mathbb{R}^{n},j\in \mathbb{N}_{0}.
\end{equation*}%
When $K=1$, the sum on the right-hand side of $\mathrm{\eqref{sum1}}$ is
interpreted as zero. Again from \cite[Lemma 3.1]{HN18},%
\begin{equation}
|\Theta _{j,m,K,k}(\partial _{k}f,V_{k})|\lesssim E_{j,m,K}(\partial
_{k}f,V_{k}),\quad 0\leq m\leq j,m,j\in \mathbb{N}_{0}. \label{imp-ine}
\end{equation}\\
\textbf{Step 2.} In this step we prove $\mathrm{\eqref{Est-th1.1}}$. For
a smooth dyadic resolution of unity $(\mathcal{F}\varphi _{j})_{j\in \mathbb{N}_{0}}$
we have%
\begin{equation*}
\Pi _{j,m,k}(\partial _{k}f,V_{k})(x)=\Lambda _{j,m}(\partial
_{k}f,V_{k})(x),\quad x\in \mathbb{R}^{n},j,m\in \mathbb{N}_{0},k\in
\{1,...,n\},
\end{equation*}%
applying Lemmas \ref{duality2} and \ref{HN1},
with the help of \eqref{duh}, we estimate the second term of \eqref{est-th1} as follows:
\begin{equation*}
\sum_{j=0}^{\infty }\sum_{m=j+1}^{\infty} \int_{\mathbb{R}^{n}}
| \Pi _{j,m,k}(\partial _{k}f,V_{k}) h_{j}(x)| \,dx
\lesssim 
\big\|\partial _{k}f \big\|_{p_{1}(\cdot )}\big\|V_k\big\|_{B_{p_{2}(\cdot ),q(\cdot
)}^{s(\cdot )}}
\end{equation*}%
for any $k\in \{1,...,n\}$. 

let $K\in \mathbb{N}$ be such that $0<s^{-}\leq s^{+}<K$, by inequality (\ref{imp-ine}) and  Lemma \ref{HN2}\ with $a=0$ we have
\begin{equation*}
\sum_{j=0}^{\infty }\sum_{m=0}^{j}\int_{\mathbb{R}^{n}} |\Theta
_{j,m,K,k}(\partial _{k}f,V_{k})(x)h_{j}(x)|dx\lesssim \big\|\partial _{k}f%
\big\|_{p_{1}(\cdot )}\big\|V_{k}\big\|_{B_{p_{2}(\cdot ),q(\cdot
)}^{s(\cdot )}},
\end{equation*}%
for any $k\in \{1,...,n\}$.

Now, we estimate the term $I_{j,m,\alpha,k}$, from the support properties of $(\mathcal{F}%
\varphi _{j})_{j\in \mathbb{N}_{0}}$, we have $\theta _{j,\alpha }\ast
\partial _{k}f=\theta _{j,\alpha }\ast \tilde{\varphi}_{j}\ast \partial
_{k}f $, where $\tilde{\varphi}_{j}=\sum_{r=-2}^{r=2}\varphi _{j+r}$ and if $j<0$
we put $\varphi _{j}=0$ . Hence, for any  $0\leq m\leq j, k\in\{1,...,n\}$ and multiindex $\alpha$,
\begin{eqnarray*}
|I_{j,m,\alpha,k}| &\lesssim &2^{m-j}(\eta _{j,N}\ast |\tilde{\varphi}_{j}\ast
\partial _{k}f|)(\eta _{m,N}\ast| \varphi _{m}\ast V_{k}|), 
\end{eqnarray*}%
 where $N>n$ is sufficiently large, therefore 
\begin{align}\label{est3}
\sum_{j=0}^{\infty }\sum_{m=0}^{j}\int_{\mathbb{R}^{n}}|I_{j,m,\alpha,k}(x) h_{j}(x)|dx \lesssim 
\int_{\mathbb{R}^{n}} \sum_{j=0}^{\infty }\eta 
_{j,N}\ast |\tilde{\varphi}_{j}\ast \partial _{k}f|(x)\vartheta_j(x) |h_{j}(x)|dx, 
\end{align}%
where 
\begin{align*}
\vartheta_j(x):=\sum_{m=0}^{j} 2^{m-j}\eta _{m,N}\ast |\varphi _{m}\ast V_{k}|(x),
 j\in \mathbb{N}_{0}, x\in \mathbb{R}^{n}.
\end{align*}
Lemmas \ref{duality2}, \ref{Alm-Hastolemma1} and  inequalities (\ref{duh}), (\ref{eneq p norm}) and (\ref{Hold-inq}) yield
\begin{align*}
\sum_{j=0}^{\infty }\sum_{m=0}^{j}\int_{\mathbb{R}%
^{n}}|I_{j,m,\alpha,k}(x) h_{j}(x)|dx &\lesssim
\big\|(2^{js(\cdot)}\eta_{j,N}\ast |\tilde{\varphi}_{j}\ast \partial _{k}f|\vartheta_j)_{j\in \mathbb{N}_{0}}\big\|_{\ell ^{q(\cdot
)}(L^{p(\cdot )})}\\
&\lesssim \sup_{k} \big\|\vartheta_k\big\|_{p_{1}(\cdot )}\\
&\quad \quad\times\big\|(\eta_{j,N_1}\ast |2^{js(\cdot)}\tilde{\varphi}_{j}\ast \partial _{k}f|(\cdot)
)_{j\in \mathbb{N}_{0}}\big\|_{\ell ^{q(\cdot
)}(L^{p_2(\cdot )})}\\
&\lesssim \big\|V_{k}\big\|_{p_{1}(\cdot )}\big\|\partial_k f\big\|_{B_{p_{2}(\cdot
),q(\cdot )}^{s(\cdot )}},
\end{align*}
this completes the proof of the the first estimate of $\mathrm{\eqref{Est-th1.1}}$ with  $A=\big\|V\big\|_{p_{1}(\cdot )}%
\big\|\nabla f\big\|_{B_{p_{2}(\cdot ),q(\cdot )}^{s(\cdot )}}$.

Regarding the second estimate of $\mathrm{\eqref{Est-th1.1}}$, for every  $j\in \mathbb{N}_{0},k\in \{1,...,n\},$ and multiindex $\alpha$ we have
\begin{equation*}
\sum_{m=0}^{j}I_{j,m,\alpha,k}=\sum_{m=0}^{j} 2^{|\alpha |(m-j)}(\theta
_{j,\alpha }\ast \tilde{\varphi}_{j}\ast \partial _{k}f)(\partial ^{\alpha
}\varphi )_{m}\ast \varphi _{m}\ast V_{k},
\end{equation*}%
and
\begin{equation*}
|\theta _{j,\alpha }\ast \tilde{\varphi}_{j}\ast \partial _{k}f|\lesssim
2^{j}\eta _{j,N_1}\ast |\tilde{\varphi}_{j}\ast f|
\end{equation*}%
for some large $ N_1>n$. Since $|\alpha|\geq 1$ we have $\partial ^{\alpha
}\varphi=\partial ^{\alpha-e_\alpha
}\partial_{i_\alpha} \varphi$ where $e_\alpha$ is the $i_\alpha$-th canonical basis vector of $\mathbb{R}^n$, hence 
\begin{equation*}
(\partial ^{\alpha}\varphi )_{m}\ast \varphi _{m}\ast V_{k}=
2^{-m} (\partial ^{\alpha-e_{\alpha}}\varphi )_{m}\ast \varphi _{m}\ast\partial_{i_\alpha} V_{k},
\end{equation*}
it follows, when $|\alpha|>1$
\begin{equation*}
\sum_{m=0}^{j}|I_{j,m,\alpha
,k}|\lesssim 
\sum_{m=0}^{j} 2^{(|\alpha |-1)(m-j)}\eta _{j,N_1}\ast |\tilde{\varphi}_{j}\ast f|
\eta _{m,N_1}\ast|\partial_{i_\alpha} V_{k}|,
\end{equation*}%
and if $|\alpha|=1$ then $\alpha =e_{\alpha}$, from the properties of $(\mathcal{F}\varphi _{j})_{j\in \mathbb{N}_{0}}$, for every $j\in\mathbb{N}_0$ we have $\sum_{m=0}^{j}\varphi_{m}\ast \varphi _{m}=(\varphi_0)_j $, hence 
\begin{align*}
\left| \sum_{m=0}^{j}  I_{j,m,\alpha
,k}\right|&=\left| (\theta
_{j,\alpha }\ast \tilde{\varphi}_{j}\ast \partial _{k}f) \big(\sum_{m=0}^{j}\varphi_{m}\ast \varphi _{m}\big)\ast \partial_{i_\alpha} V_{k}\right|\\
&\lesssim \eta _{j,N_1}\ast |\tilde{\varphi}_{j}\ast f|\eta _{j,N_1}\ast|\partial_{i_\alpha} V_{k}|,
\end{align*}%
then with similar arguments as for (\ref{est3}) with $I_{j,m,\alpha ,k}$, for every multiindex $\alpha$ we have 
\begin{align*}
\sum_{j=0}^{\infty }\sum_{m=0}^{j}\int_{\mathbb{R}%
^{n}}|I_{j,m,\alpha,k}(x) h_{j}(x)|dx \lesssim \big\|\partial_{i_\alpha} V_{k}\big\|_{p_{1}(\cdot )}\big\| f\big\|_{B_{p_{2}(\cdot),q(\cdot )}^{s(\cdot )}},
\end{align*}
this proves the second estimate of $\mathrm{\eqref{Est-th1.1}}$ with  $A=\big\|\nabla V\big\|_{p_{1}(\cdot )}\big\| f\big\|_{B_{p_{2}(\cdot),q(\cdot )}^{s(\cdot )}}$.\\
\textbf{Step 3.} In this step we prove $\mathrm{\eqref{Est-th1.1.1}}$, for every $x\in \mathbb{R}^{n},j,m\in \mathbb{N}_{0}$ and $k\in\{1,...,n\}$,
\begin{align*}
\Pi _{j,m,k}(\partial _{k}f,V_{k})(x)&=
\varphi_j*(\partial_kf)\Psi_m*\varphi_m*V_k-\varphi_j*(\partial_kf\Psi_m*\varphi_m*V_k)\\
&= 2^j((\partial_k\varphi)_j*f)\Psi_m*\varphi_m*V_k 
-2^j(\partial_k\varphi)_j*(f\Psi_m*\varphi_m*V_k)\\
&\hspace{7cm}+\varphi_j*(f \Psi_m*\varphi_m*\partial_kV_k) 
\\
&=J_{j,m,k}^{1}(\partial
_{k}f,V_{k})(x)+J_{j,m,k}^{2}(\partial _{k}f,V_{k})(x),
\end{align*}%
where%
\begin{equation*}
J_{j,m,k}^{1}(\partial _{k}f,V_{k})(x):=\varphi_j*(f \Psi_m*\varphi_m*\partial_kV_k) 
\end{equation*}%
and%
\begin{equation*}
J_{j,m,k}^{2}(\partial _{k}f,V_{k})(x):=\int_{\mathbb{R}^{2n}}2^{j}(\partial
_{k}\varphi )_{j}(x-y)(\Psi _{m}(x-z)-\Psi _{m}(y-z))f(y)\varphi _{m}\ast
V_{k}(z)dydz.
\end{equation*}%
It follows,
\begin{equation*}
\left\|
\left( 2^{js(\cdot )}\sum_{m=0}^{\infty
}\sum_{k=1}^{n}J_{j,m,k}^{1}(\partial _{k}f,V_{k})\right)_{j\in \mathbb{N}_{0}}%
\right\|_{\ell ^{q(\cdot)}(L^{p(\cdot )})}=\big\|f\mathrm{div}(V)\big\|%
_{B_{p(\cdot ),q(\cdot )}^{s(\cdot )}}.
\end{equation*}%
we see that $J_{j,m,k}^{2}(\partial _{k}f,V_{k})=2^{j}\Lambda
_{j,m}(f,V_{k})$ but with $(\partial _{k}\varphi )_{j}$ in place of $\varphi
_{j}$. Using the same type of arguments as in Step 2 we see that%
\begin{align*}
\left\|\left(2^{js(\cdot )}\sum_{m=0}^{\infty
}\sum_{k=1}^{n}J_{j,m,k}^{2}(\partial _{k}f,V_{k})\right)_{j\in \mathbb{N}_{0}}%
\right\|_{\ell ^{q(\cdot)}(L^{p(\cdot )})}
 &\lesssim \big\|\nabla V\big\|_{p_{1}(\cdot )}\big\|f\big\|_{B_{p_{2}(\cdot
),q(\cdot )}^{s(\cdot )}}
\\
&\hspace{2.5cm}+\big\|f\big\|_{p_{1}(\cdot )}\big\|V\big\|%
_{B_{p_{2}(\cdot ),q(\cdot )}^{s(\cdot )+1}}.
\end{align*}%
The proof is completed.
\end{proof}

The next statement is an improvement of $\mathrm{\eqref{Est-th1.1}}$ 
with $0<s^{-}\leq s^{+}<1$ and of $\mathrm{%
\eqref{Est-th1.1.1}}$ with $-1<s^{-}\leq s^{+}<0$.
\begin{theorem}
\label{result2}Let $s\in C_{\mathrm{loc}}^{\log }(\mathbb{R}^{n}),p,p_{1},p_{2}\in 
\mathcal{P}^{\log }(\mathbb{R}^{n})$  and $q\in \mathcal{P} (\mathbb{R}^{n})$ with $\frac{1}{q}\in C_{\mathrm{loc}}^{\log }(\mathbb{R}^{n}) $ and $1/p=1/p_1+1/p_2$. Let $V=(V_{1},...,V_{n})\in \left( \mathcal{S}(\mathbb{R}^{n})\right) ^{n}$
be a vector field. Then for any $f\in \mathcal{S}(\mathbb{R}^{n})$%
\begin{equation*}
\big\|( 2^{js(\cdot )}[V\cdot \nabla ,\Delta _{j}]f )_{j\in \mathbb{N}_{0}}\big\|_{\ell ^{q(\cdot)}(L^{p(\cdot )})}
\lesssim \big\|\nabla f%
\big\|_{p_{1}(\cdot )}\big\|V\big\|_{B_{p_{2}(\cdot ),q(\cdot )}^{s(\cdot
)}}
\end{equation*}%
if  $0<s^{-}\leq s^{+}<1$, and%
\begin{equation*}
\big\|( 2^{js(\cdot )}[V\cdot \nabla ,\Delta _{j}]f )_{j\in \mathbb{N}_{0}}\big\|_{\ell ^{q(\cdot)}(L^{p(\cdot )})}\lesssim \big\|f\mathrm{div}%
(V)\big\|_{B_{p(\cdot ),q(\cdot )}^{s(\cdot )}}+\big\|f\big\|_{p_{1}(\cdot )}%
\big\|V\big\|_{B_{p_{2}(\cdot ),q(\cdot )}^{s(\cdot )+1}}
\end{equation*}%
if $-1<s^{-}\leq s^{+}<0.$
\end{theorem}
\begin{proof}
 The first estimate follows by Steps 1-2 of Theorem \ref%
{result1}, with $K=1$ and $a=0$, while the second one follows by the same
arguments of Step 3 in Theorem \ref{result1}.
\end{proof}

The next theorem presents commutator estimates with various indices for variable  Besov spaces.
\begin{theorem}
\label{result3}
Let $s,s_1,s_2\in C_{\mathrm{loc}}^{\log }(\mathbb{R}^{n}) ,p,p_{1},p_{2}\in 
\mathcal{P}^{\log }(\mathbb{R}^{n})$ and $q,q_1,q_2 \in \mathcal{P}(\mathbb{R}^{n}) $ with
 $1/q,1/q_1,1/q_2\in C_{\mathrm{loc}}^{\log }(\mathbb{R}^{n}) $, $s=s_1+s_2, s^{-}>0, s_2^+<1$, $1/p=1/p_1+1/p_2$ and $1/q=1/q_1+1/q_2$. Then for any $f\in \mathcal{S}(\mathbb{R}^{n})$%
\begin{align*}
\big\|( 2^{js(\cdot )}[V\cdot \nabla ,\Delta _{j}]f )_{j\in \mathbb{N}_{0}}\big\|_{\ell ^{q(\cdot)}(L^{p(\cdot )})}
\lesssim 
 \big\|\nabla f\big\|_{p_{1}(\cdot )}\big\|V\big\|_{B_{p_{2}(\cdot ),q(\cdot )}^{s(\cdot
)}} + 
\big\|\nabla f \big\|_{B_{p_{1}(\cdot ),q_1(\cdot )}^{s_1(\cdot)}}
\big\|V \big\|_{B_{p_{2}(\cdot ),q_2(\cdot )}^{s_2(\cdot)}}.
\end{align*}%
\end{theorem}
\begin{proof} We employ the results of  Step 2 of Theorem \ref{result1}, it suffices to estimate $I_{j,m,\alpha,k}$ 
. By the estimate (\ref{est3}) with $s=s_1+s_2$ and Lemma \ref{DHR-lemma}, for every $x\in \mathbb{R}^{n},k\in
\{1,...,n\}$ and multiindex $\alpha$, 
\begin{align*}
\sum_{j=0}^{\infty }\sum_{m=0}^{j}\int_{\mathbb{R}%
^{n}}|I_{j,m,\alpha,k}(x) h_{j}(x)|dx &\lesssim 
\int_{\mathbb{R}^{n}} \sum_{j=0}^{\infty }\eta 
_{j,N_1}\ast |2^{js_1(\cdot)}\tilde{\varphi}_{j}\ast \partial _{k}f|(x)\\
&\hspace{3cm}\times 2^{js_2(x)}\vartheta_j(x) 2^{-js(x)}|h_{j}(x)|dx,
\end{align*}%
and for every $x\in \mathbb{R}^{n}$ and $j\in\mathbb{N}_0$, we have
\begin{align*}
2^{js_2(x)}\vartheta_j(x)&=\sum_{m=0}^{j} 2^{(m-j)(1-s_2(x))}2^{m s_2(x)}\eta _{m,N}\ast |\varphi _{m}\ast V_{k}|(x)\\
&\lesssim 
\sum_{m=0}^{j} 2^{(m-j)(1-s_2^+)}\eta _{m,N_2}\ast |2^{m s_2(\cdot)}\varphi _{m}\ast V_{k}|(x),
\end{align*}
 where $N_1,N_2>n$ are sufficiently large.  Since $s_2^+<1$, by Lemmas \ref{thm compar} and \ref{Alm-Hastolemma1},
\begin{equation*}
\big\|(2^{js_2(\cdot)} \vartheta_j(\cdot))_{j\in \mathbb{N}_{0}}\big\|_{\ell ^{q_2(\cdot)}(L^{p_2(\cdot )})}\lesssim 
\big\|V_k \big\|_{B_{p_{2}(\cdot ),q_2(\cdot )}^{s_2(\cdot)}},
\end{equation*}
by Lemma \ref{duality2} and inequalities (\ref{duh}), (\ref{Hold-GEN}), for every $k\in\{1,...,n\}$ we have
\begin{align*}
\sum_{j=0}^{\infty }\sum_{m=0}^{j}\int_{\mathbb{R}%
^{n}}|I_{j,m,\alpha,k}(x)h_{j}(x)|dx &\lesssim 
\big\|(2^{js_2(\cdot)} \vartheta_j(\cdot) \eta_{j,N_1}\ast |2^{js_1(\cdot)}\tilde{\varphi}_{j}\ast \partial _{k}f|
)_{j\in \mathbb{N}_{0}}\big\|_{\ell ^{q(\cdot
)}(L^{p(\cdot )})}\\
&\lesssim 
\big\|(2^{js_2(\cdot)} \vartheta_j(\cdot))_{j\in \mathbb{N}_{0}}\big\|_{\ell ^{q_2(\cdot)}(L^{p_2(\cdot )})}\\
&\hspace{1.5cm}\times\big\|(
\eta_{j,N_1}\ast |2^{js_1(\cdot)}\tilde{\varphi}_{j}\ast \partial _{k}f|
)_{j\in \mathbb{N}_{0}}\big\|_{\ell ^{q_1(\cdot)}(L^{p_1(\cdot )})}\\
&\lesssim \big\|\partial _{k}f \big\|_{B_{p_{1}(\cdot ),q_1(\cdot )}^{s_1(\cdot)}}
\big\|V_k \big\|_{B_{p_{2}(\cdot ),q_2(\cdot )}^{s_2(\cdot)}}.
\end{align*}%
\end{proof}

  Corresponding statements to Theorems \ref{result1}, \ref{result2} and \ref{result3}  were presented in theorems
1.1, 1.2, and 1.3 of \cite{HN18} for classical Triebel–Lizorkin and Besov spaces $F_{p,q}^{s}$ and $B_{p,q}^{s}$.
 In \cite[Section 6.3]{HN18} the corresponding results of Theorems \ref{result1}, \ref{result2} and \ref{result3} were presented for variable exponent  Triebel–Lizorkin and Besov spaces $F_{p(\cdot ),q}^{s}$ and $B_{p(\cdot ),q}^{s}$ under the assumptions that $s$ is  constant and $q\in[1,\infty]$ is  constant  with $p^->1$ and $p^+<\infty$. 
For variable Triebel–Lizorkin spaces $F_{p(\cdot ),q(\cdot)}^{s(\cdot)}$,
 the corresponding theorems to Theorems \ref{result1} and \ref{result2} were presented in theorems 1 and 2 of  \cite{art1}, respectively. An extension of these estimates to the general case $s\in C_{\mathrm{loc}}^{\log }(\mathbb{R}^n)$ is still open.

\end{document}